\documentclass[12pt,draftcls,onecolumn]{IEEEtran}

\usepackage{amssymb,amsmath,amsfonts,amsthm}
\usepackage{algorithm,algorithmic}
\usepackage{cite}
\usepackage{color}
\usepackage{epsfig}
\usepackage{float}
\usepackage{graphicx}
\usepackage{graphics}
\usepackage{url}
\usepackage{subfigure}
\usepackage{xspace}
\usepackage{setspace}

\newtheorem{lemma}{Lemma}

\newtheorem{theorem}{Theorem}

\newtheorem{remark}{Remark}
\newtheorem{defn}{Definition}

\floatstyle{ruled}
\newfloat{model}{H}{mod}
\floatname{model}{\footnotesize Model}
\newfloat{notatio}{H}{not}
\floatname{notatio}{\footnotesize Notation}

\usepackage{url,changebar,bm,xspace,dsfont}
\let\mathbb=\mathds 
\def\diag{\mathop{\mathrm{diag}}}  
\def\det{\mathop{\mathrm{det}}}  
\def\Co{\mathop{\mathrm{Co}}}  
\def\d{\mathrm{d}} 

\ifCLASSINFOpdf
 \else
\fi

\begin{document}

\title{A Distributed Power Control and Transmission Rate Allocation Algorithm over Multiple Channels}

\author{Themistoklis Charalambous,~\IEEEmembership{Member,~IEEE}
\thanks{Themistoklis Charalambous is with the Department of Electrical and Computer Engineering, University of Cyprus, Nicosia 1678 (Email: themis@ucy.ac.cy).}
}

\maketitle

%
%
\begin{abstract}
\noindent In this paper, we consider multiple channels and wireless nodes with multiple transceivers. Each node assigns one transmitter at each available channel. For each assigned transmitter  the node decides the power level and data rate of transmission in a distributed fashion, such that certain Quality of Service (QoS) demands for the wireless node are satisfied. More specifically, we investigate the case in which the average SINR over all channels for each communication pair is kept above a certain threshold. A joint distributed power and rate control algorithm for each transmitter is proposed that dynamically adjusts the data rate to meet a target SINR at each channel, and to update the power levels allowing for variable desired SINRs. The algorithm is fully distributed and requires only local interference measurements. The performance of the proposed algorithm is shown through illustrative examples.
\end{abstract}




\ifCLASSOPTIONpeerreview
\begin{center} \bfseries EDICS Category: 3-BBND \end{center}
\fi
\IEEEpeerreviewmaketitle

%
%
\section{Introduction}\label{introduction}

Throughput is an important metric in wireless networks and its improvement is achievable by efficient use of the wireless channel. Since each transmission corresponds to a spatiotemporal propagation of radio waves that are received by all nodes in proximity utilising the same channel, nodes interfere with each other when they use the same channel simultaneously. This is called co-channel interference (CCI) \cite{2002_rappaport_wireless}. In addition, the power of each transmitter in a wireless network is directly related to the resource usage of the link and it is a valuable resource, since the batteries of the wireless nodes have limited lifetime. Increased power ensures longer transmission distance and higher data transfer rate. However, power minimisation not only increases battery lifetime, but also increases the effective interference mitigation which in turn, increases the overall network capacity by allowing higher frequency reuse. Furthermore, the near-far problem \cite{2002_rappaport_wireless} is commonly solved by dynamic power adjustments by the transmitters. Dynamic power control in wireless networks allows devices to setup and maintain wireless links with minimum power while satisfying certain constraints on QoS. As a result, power control (also referred to as transmission power control) has been a prominent research area with increased interest during the last two decades.


In this paper, we combine power control with transmission rate in the presence of multiple channels. The main goal for each communication pair is to retain the average of the SINRs (or the transfer rates) from all the available channels above or equal to a certain threshold. A distributed scheme is developed in which individual users can cooperate in such a way that otherwise infeasible states can be achieved in a distributed manner without the need of interlink communication among users and centralized computation as required in a centrally controlled wireless environments. In order to be able to reach states that are infeasible to be achieved by simultaneous transmissions, it is necessary to allow for variable desired SINRs to be assigned in each of the available channels. Hence, we relax the hard constraints for SINRs and try to achieve the desired QoS on average by making use of multiple channels. As a result, the problem targeted in this paper is how communication pairs can achieve on average the QoS targets that belong in the convex hull of the system's feasible set, but would probably be impossible to be achieved with the network setting by simultaneous transmissions. The difference of the problem being solved in this study compared to other related studies (e.g. for opportunistic transmission and Multi-Input Multi-Output (MIMO) interference systems) relies on the fact that we consider different frequency channels in which the network configuration might be very different, and also consider the general SINR regime in which the problem cannot be transformed into a tractable convex optimisation problem. 
As a result, even though a considerable amount of work has been done to characterize the capacity of the system and to find the maximum throughput of the system when multiple channels are available, to the best of our knowledge there is no work on the joint power control and transmission rate allocation for multiple channels, when multiple users appear in the network and they need to fulfil a QoS requirement in the general SINR regime.


Early work in the field of power control for wireless networks \cite{1973:Aein,1992:zander} proposed power balancing, which equalizes the Signal-to-Interference Ratio (SIR) in all the wireless links. These algorithms need global information about the network setting. This capacity improvement initiated extensive research on power control with focus on the design of distributed algorithms (working synchronously \cite{1992:zander_Distributed, 1993:foschini}, asynchronously  \cite{1993:Mitra, 1995:Yates}, under constraints \cite{1995:grandhi_constrained}, with active link protection \cite{2000:bambos_channel}, in the presence of time delays \cite{2008:constantdelays, 2009:varyingdelays,2009:myThesis,2009:Moller,2010:Moller} etc.) to meet a prefixed SINR target (hard constraint), determined by the QoS requirements. The prefixed SINR target tracking (which is the condition for inelastic traffic) ensures that a constant transmission rate can be sustained. If a feasible solution exists, then there exists a unique solution that minimizes transmit power in a pareto sense. But, if not, then the performance of the whole network degrades and the capacity is deteriorated. The target tracking approach is suitable for real-time, delay-sensitive applications like mobile phone services. 

In view of the proliferation of wireless data though, it is essential to investigate further transmission schemes, that is, techniques that facilitate elastic and/or opportunistic traffic should be considered, where time-varying rates are allowed and large delays are tolerated, such as in the World Wide Web (WWW) and video streaming. When each node is assigned its desired SINR, it has no knowledge of the network and as a result, the combination of all users' QoS requirements are sometimes impossible to be fulfilled in a single channel. If nodes could exchange information and obtain full information about the network setting, then they could adjust their desired SINRs so that the network becomes globally asymptotically stable when all transmitters operate simultaneously. Since every node affects all the other nodes in the network, it is not possible (especially in big networks) to acquire the knowledge required about the whole network. Furthermore, even in small networks consisting of two communication pairs only (i.e. four nodes), simultaneous transmissions for the desired SINRs may not be feasible. 

In general, the fluctuation of wireless channels can be exploited using power and transmission rate control in order to meet QoS requirements, i.e., a node can increase its transmit power whenever the interference at its receiver is low and decrease it when the interference is high. The adaptation of transmission rate to the actual propagation and interference conditions is facilitated by modern adaptive modulation and coding schemes. The transmission rate can be selected so that the Bit Error Rate (BER) is sufficiently small and this can be achieved by adjusting the transmission rates according to the SINR. In that way, more information is transmitted when the channel conditions are favourable by adjusting the transmission rate accordingly. This approach enables the improvement of the network's convergence and the satisfaction of heterogeneous service requirements. 

It is well known that the adaptation of the transmit power, data rate, and coding scheme increases spectral efficiency. The IEEE 802.11b scheme allows nodes to increase their transfer rate up to $11$ Mbps, depending on the SINR at the receiver. The performance achieved through power control can be further improved by allowing for dynamic adjustment of the transfer rate based on the SINR. There exist some approaches in the literature that deal with dynamic adjustment of the desired SINR, and hence the data rate. Some consider a single channel only (e.g. \cite{1999:Qiu,2000:Wirelessperformanceanalysis, 2003:Xiao, 2009:eccpaper}) whereas others consider the multi-channel case (e.g. \cite{2000:Wuintelligentmedium,2006:Siam, 2006:distributed_compensation}) as well. Concepts and related work on both approaches are briefly described in turn.  

Firstly, we describe the approach that supports the use of adaptive desired SINR for the single channel only. Some schemes (e.g. \cite{2003:Xiao, 2005:subramanian} and \cite{2009:eccpaper}) tried to deal with the case in which the desired SINR could be adaptive, depending on the channel conditions, but the main idea is restricted within the limits of a compromise on the QoS demand. Their approach enables the improvement of the system convergence and the satisfaction of heterogeneous service requirements. 
In \cite{2004:Sung_Leung, 2006:Leung} an opportunistic power control algorithm is proposed which stems from the concepts of opportunistic communications \cite{1998:Tse_multi-accessfading, 1998:Hanly_multi-accessfading}.  This algorithm provides tunable parameters in order to exploit the trade-off between throughput and power consumption. The algorithm is proven to converge to a unique fixed point. However, it does not consider a QoS requirement, but rather tries to maximize the throughput for individual users. The same problem has also been targeted by \cite{2010:fabiano}, in which a class of distributed power control algorithms that exploit the channel variability opportunistically is presented.  These algorithms are restricted to maintain the desired QoS in a certain range, which is specified by the node. Hence, these algorithms can be converted to a conventional algorithm that requires a prefixed QoS requirement.  Nevertheless, none of the approaches associate the QoS requirement with a cost or utility function that leads to a solvable power control problem. Consequently, there is no way to guarantee that on average their QoS targets are fulfilled. In other words, there is no way to maintain a constant average data rate.

Since wireless nodes may use different channels and they have different locations at different times, resource allocation can take advantage of the diversity in space, time and frequency. Wireless technology standards provide a radio-frequency (RF) spectrum with a set of many non-overlapping channels  and a node has the option to choose on which channels to transmit. Therefore, we can make use of extra channels, so that nodes experiencing much interference by other nodes can choose to operate in a different channel, if possible. However, since it is possible to establish diversity by introducing multiple transceivers or transceivers with multiple antennae on a single node, communication between two wireless nodes should not be confined on a single channel only, but make use of all the available channels. Multi-Input Multi-Output (MIMO) links are antenna arrays at both ends of a link and transmit parallel streams at the time and frequency channel. Channels arising from the use of spatial diversity at both the transmitter and the receiver have been considered in the literature (e.g.  \cite{1999:Teletar} and reference there in) and it is shown to have improved capacity with respect to Single-Input Single-Output (SISO) systems because of the use of parallel channels; such a mutli-access network with MIMO links has been considered in previous studies (such as \cite{2006:Arslan, 2008:ScutariPalomar, 2008:loyka_capacity} and references therein). As aforementioned there is a significant difference of the problem being solved in this study and MIMO interference systems. In this study, we consider different frequency channels in which the network configuration might be very different and also consider the general SINR regime in which the problem cannot be transformed into a tractable convex optimisation problem. 
Even though a considerable amount of work has been done to characterize the capacity of the system and to find the maximum throughput of the system when multiple channels are available, to the best of our knowledge there is no work on the joint power control and transmission rate allocation for multiple channels when multiple wireless nodes appear in the network and they need to fulfil a QoS requirement in the general SINR regime.


The rest of the paper is organized as follows. Section \ref{Notation} establishes the notation used throughout the paper. Section \ref{model} describes the mathematical system model. Section \ref{preliminaries} presents some preliminary results on the conditions for feasible networks and provides sufficient conditions for stability of the Foschini-Miljanic (FM) algorithm \cite{1993:foschini} using Lyapunov Stability theory, a useful preliminary result for the subsequent analysis. Section \ref{feasibilityRegion} addresses the feasibility regions based on the channel model considered. Section \ref{formulation} formulates the problem being targeted. Section \ref{results} derives the main result of this work, in which we develop a distributed scheme that allows variable desired SINRs and facilitate the use of multiple channels. Section \ref{examples} demonstrates the validity of the proposed scheme through illustrative examples. Finally, Section \ref{conclusions} concludes the paper with a brief discussion on the algorithm developed as well as future directions.

%
%
\section{Notation}\label{Notation}

The sets of complex, real and natural numbers are denoted by $\mathds{C}$,  $\mathds{R}$ and $\mathds{N}$, respectively; their positive orthant is denoted by the subscript $+$ (e.g. $\mathds{C}_{+}$). Vectors are denoted by bold letters whereas matrices are denoted by capital letters.  $A^{T}$ and $A^{-1}$ denote the transpose and inverse of matrix $A$ respectively. For two symmetric matrices $A$ and $B$, $A \succ ( \succeq)B$ means that $A-B$ is (semi-)positive definite. By $I$ we denote the identity of a squared matrix. $|A|$ is the elementwise absolute value of the matrix (i.e. $|A|\triangleq[|A_{ij}|]$), $A(<) \leq B$ is the (strict) element-wise inequality between matrices A and B. A matrix whose elements are nonnegative, called nonnegative matrix, is denoted by $A \geq 0$ and a matrix whose elements are positive, called positive matrix, is denoted by $A>0$. $\sigma(A)$ denotes the spectrum of matrix $A$, $\lambda(A)$ denotes an eigenvalue of matrix $A$, and $\rho(A)$ denotes its spectral radius. $\det(A)$ denotes the determinant of a squared matrix $A$ and $\diag(x_{i})$ the matrix with elements $x_{1}$, $x_{2}$ , $\ldots$ on the leading diagonal and zeros elsewhere.

Further notation used in the paper is tabulated below (Table \ref{notation}):
\begin{notatio}\caption{Notation used in the paper:}
\begin{tabular}{l|l}
$\mathcal{N}$ & The set of all nodes in the network\\
$\mathcal{L}$ & The set of active links in the network\\
$\mathcal{G}$ & The graph of the network\\
$\mathcal{T}$ & The set of transmitters in the network\\
$\mathcal{R}$ & The set of receivers in the network\\
$\mathcal{C}$ & The set of available channels\\
$g_{ij}$ & The channel gain on the link $i \rightarrow j$\\
$p_{i,k}$ & The power level of transmitter $i$ in channel $k$\\
$w_{i,k}$ & The interference at receiver $i$ in channel $k$\\
$x_{i,k}$ & The desired SINR at receiver $i$ in channel $k$\\
$\nu_{i,k}$ & The variance of thermal noise at receiver $i$ in channel $k$\\
$\gamma_i$ & The average capture ratio desired at the $i^{th}$ receiver\\
\end{tabular}\label{notation}
\end{notatio}

%
%
\section{System Model}\label{model}

The system model can be divided into two levels: level 1 describing the network as a whole; and level 2 describing the channels. Thus, we have the network model and the channel model. At the network level, the model concerns the general topology of the nodes and their characteristics. At the channel level, the model describes the assessment of the link quality between communication pairs and the interaction between the nodes in the network.

\subsection{Network Model}
Consider a network where $\mathcal{T}$ denotes the set of transmitters and $\mathcal{R}$ denotes the set of receivers in the network. The links are assumed to be unidirectional and each node is supported by omnidirectional antennae. 
At each time instant, each node can act as a receiver or a transmitter only due to the half-duplex nature of the wireless transceiver. Each transmitter aims at communicating with a single node (receiver) only. 

\subsection{Channel model}
As mentioned above, the link quality is measured by the SINR. The channel gain on
the link between transmitter $i$ and receiver $j$ is denoted by
$g_{ij}$ and incorporates the mean path-loss as a function of
distance, shadowing and fading, as well as cross-correlations
between signature sequences. All the $g_{ij}$'s are positive and can
take values in the range $(0,1]$ (see Figure \ref{network_example}). The power level chosen by
transmitter $i$ is denoted by $p_{i}$ and the intended receiver is
also indexed by $i$. $\nu_{i}$ denotes the variance of thermal noise
at the receiver $i$, which is assumed to be additive Gaussian noise.

The interference power at the $i$th receiver consists of both the interference caused by other transmitters in the network $\sum_{j\in \mathcal{T}_{-i}}g_{ji}p_{j}$ (where $\mathcal{T}_{-i}$ denotes all the transmitters $j$ in the network that interfere with transmitter's $i$ communications, i.e. $j\neq i, j \in \mathcal{T}$), and the thermal noise $\nu_i$ in node $i$'s receiver. Therefore, the interference at the receiver $i$ is given by
\begin{align}\label{interference}
w_{i}=\sum_{j\in \mathcal{T}_{-i}}{g_{ji}p_{j} +\nu_{i}},
\end{align}
while the SINR\index{SINR} at the receiver $i$ is given by
\begin{align}\label{SINR1}
\Gamma_{i}=\frac{g_{ii}p_{i}}{\sum_{j\in \mathcal{T}_{-i}}{g_{ji}p_{j} + \nu_{i}}}\ .
\end{align}
Due to the unreliability of the wireless links, it is necessary to
ensure QoS in terms of SINR in wireless
networks. Hence, independently of nodal distribution and traffic
pattern, a transmission from transmitter $i$ to its corresponding
receiver is successful (error free) if the SINR of the receiver is
greater or equal to $\gamma_{i}$ ($\Gamma_{i} \geq \gamma_{i}$),
called the \textit{capture ratio} which depends on the modulation
and coding characteristics of the radio. Therefore,
\begin{align}\label{SINR_geq}
\frac{g_{ii}p_{i}}{\sum_{j\in \mathcal{T}_{-i}}{g_{ji}p_{j} +
\nu_{i}}}\geq \gamma_{i}
\end{align}

\begin{figure}[h]
    \centering
  \includegraphics[width=2.3in]{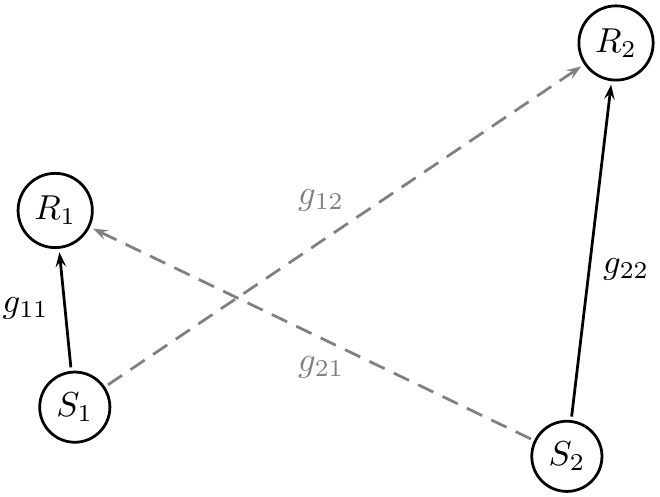}
  \caption{An example of a network consisting of two communication pairs only.  Each pair $i$ consists of a transmitter $S_i$ and a receiver $R_i$ connected with a solid line while the grey dotted arrows indicate the interference that transmitters cause to the neighboring receivers.}
  \label{network_example}
\end{figure}


\begin{remark}
We have not specified any model for determining the positions of the nodes, since we investigate the general case of a network that any position could be possible. Further, we have also not specified any model related to the propagation of signals. In our context, these two models will ultimately specify the channel gains, $g_{ij}$. Nevertheless, they are of secondary importance in this study since it is focused on how the QoS is improved given the channel gains and it only depends on the power of the received signals. Note also that the effect of nodes' mobility is not consider in this study. However, this could be relaxed to the case of low mobility, where the link structure is expected to change slowly with respect to the packet rates and network updates.
\end{remark}

%
%
\section{Preliminaries}\label{preliminaries}

\subsection{System Feasibility}

Inequality \eqref{SINR_geq} which depicts the QoS requirement of a communication pair $i$ while transmission takes place is equivalent to the following condition:
\begin{align}
p_{i} \geq \gamma_{i}\left(\sum_{j\in \mathcal{T}_{-i}}{\frac{g_{ji}}{g_{ii}}p_{j} + \frac{\nu_i}{g_{ii}}}\right) .
\end{align}
In matrix form, for a network consisting of $n$ communication pairs, this can be written as
\begin{align}\label{SINRmatrix6}
\mathbf{p} \geq \Gamma G \mathbf{p} + \bm{\eta},
\end{align}
where $\Gamma=\diag(\gamma_{i})$, $ \mathbf{p}=\left(
\begin{array}{cccc}
p_{1} & p_{2} & \ldots & p_{n} \\
\end{array}
\right)^{T} $, $ \eta_{i}={\gamma_{i}\nu_i}/{g_{ii}}$ and
\begin{align*}
G_{ij}=\begin{cases}
0 & \text{, if $i=j$},\\
\frac{g_{ji}}{g_{ii}} & \text{, if $i \neq j$}.
\end{cases}
\end{align*}
Let $C=\Gamma G$ such that
\begin{align}\label{C}
C_{ij}=\begin{cases}
0 & \text{, if $i=j$},\\
\gamma_{i}\frac{g_{ji}}{g_{ii}} & \text{, if $i \neq j$},
\end{cases}
\end{align}
then \eqref{SINRmatrix6} can be written as
\begin{align}\label{SINRmatrix16}
(I-C)\mathbf{p} \geq \bm{\eta}.
\end{align}

\noindent The matrix $C$ has strictly positive elements off diagonal and it is reasonable to assume that is irreducible \cite{1985:matrix}, since we are not considering totally isolated groups of links that do not interact with each other. By the Perron-Frobenius theorem \cite{1985:matrix}, we have that the spectral radius of the matrix $C$ is a simple eigenvalue, while the corresponding eigenvector is positive componentwise. The necessary and sufficient condition for the existence of a nonnegative solution to inequality \eqref{SINRmatrix16} for every positive vector $\bm{\eta}$ is that $(I-C)^{-1}$ exists and is nonnegative. However, $(I-C)^{-1} \geq 0$ if and only if $\rho(C)<1$ \cite{1994:Topicsmatrix} (\emph{Theorem 2.5.3}), \cite{2005:PerronFrobenius}. Therefore, the necessary and sufficient condition for \eqref{SINRmatrix16} to have a positive solution $\mathbf{p^{*}}$ for a positive vector $\displaystyle \bm{\eta}$ is that the Perron-Frobenius eigenvalue of the matrix $C$ is less than $1$. Hence, a network described by matrix $C$ is feasible if and only if $\rho(C)<1$.

\subsection{Stability of the FM algorithm using Lyapunov Theory}\label{FMlyap}

\noindent Given a network, we know that a feasible solution to the power control problem exists if the Perron-Frobenius eigenvalue of matrix $C$ is less than one ($\rho(C)<1$), where matrix C is given by \eqref{C}. If the condition holds, the FM algorithm \cite{1993:foschini} is asymptotically stable. Since the FM algorithm is linear, when the stability condition holds we know that there exists a quadratic Lyapunov function which establishes the stability of the system.

\noindent  Now, the notions of Lyapunov and D-stability are being used for the derivation of some preliminary results on the stability of the FM algorithm, defined by the following differential equation \cite{1993:foschini}:
\begin{align}\label{Pupdate}
	\frac{\mathrm{d} p_{i}(t)}{\mathrm{d} t}=k_{i} \left(-p_{i}(t)+\gamma_{i}\left({\sum_{j\in \mathcal{N}_{-i}}{\frac{g_{ji}}{g_{ii}}p_{j}(t) + \frac{\nu_i}{g_{ii}}}}\right)\right)
\end{align}
where $k_{i} \in \mathbb{R}_+$ denotes the proportionality constant and $\gamma_{i}$ denotes the desired SINR. It is assumed that each node $i$ has only knowledge of the interference at its own receiver.

\noindent  By defining $\mathbf{e}(t)=\mathbf{p}(t)-\mathbf{p}^*$, where $\mathbf{p}^*$ denotes the vector of power levels when the system is at the equilibrium, the stability of the system will be assessed. With the above system of equations in mind, we note the following important result that is central to the conclusion of this paper.

\begin{lemma}\label{chap5_lemma}
Suppose the spectral radius of matrix $C$ in \eqref{C} is less than $1$,
then the quadratic Lyapunov function
\begin{align}
V(\mathbf{e}) = \mathbf{e}^{T}A^{T}DA \mathbf{e} ,
\label{lyapfunction}
\end{align}
where $A = KH$, $K = \diag(k_{i})$, $H=I-C$ and $D$ is a positive diagonal matrix, stabilizes \eqref{Pupdate}
{\small
\begin{align}
\frac{\d p_{i}(t)}{\d t}=k_{i} \left(-p_{i}(t)+\gamma_{i}\left({\sum_{j\in \mathcal{T}_{-i}}{\frac{g_{ji}}{g_{ii}}p_{j}(t) + \frac{\nu_i}{g_{ii}}}}\right)\right)
\label{chap5_updateAlgorithm}
\end{align}}
for $\gamma_{i}$, $g_{ji}$, $\nu_i>0$, for any initial state $p_{i}(0)>0$ and for any proportionality constant, $k_{i}>0$.
\end{lemma}

\noindent  \begin{proof}
The power control algorithm \eqref{chap5_updateAlgorithm} can be written in matrix form as
\begin{align}\label{FMmatrixforminitial}
\mathbf{\dot{p}}(t)=-KH\mathbf{p}(t)+K\bm{\eta}
\end{align}
where $K = \diag(k_{i})$ and
\begin{align}
H_{ij}=\begin{cases}
1 & \text{, if $i=j$},\\
-\gamma_{i}\frac{g_{ji}}{g_{ii}} & \text{, if $i \neq j$}.
\end{cases}
\end{align}

\noindent Let $A=KH$, $\mathbf{q}=K\bm{\eta}$ and $\mathbf{e}(t)=\mathbf{p}(t)-\mathbf{p}^{*} \Rightarrow \mathbf{\dot{e}}(t)=\mathbf{\dot{p}}(t)$. Then, \eqref{FMmatrixforminitial} can be expressed as
\begin{align}\label{FMmatrixforminitial1}
\mathbf{\dot{p}}(t)=-A\mathbf{p}(t)+\mathbf{q} .
\end{align}
At equilibrium, $\mathbf{\dot{p}}=0$ and hence $A\mathbf{p}^{*}=\mathbf{q}$. Therefore, \eqref{FMmatrixforminitial1} is written as
\begin{align}
\mathbf{\dot{p}}(t) = \mathbf{\dot{e}}(t) =  -A\mathbf{p}(t)+ A\mathbf{p}^{*} = -A(\mathbf{p}(t)-\mathbf{p}^{*}) = - A \mathbf{e}(t),
\end{align}
giving $\mathbf{\dot{e}}(t) = - A \mathbf{e}(t)$. Consider the candidate Lyapunov function
\begin{align}\label{lyapcandidate}
V(\mathbf{e}) = \mathbf{e}^{T}A^{T}DA \mathbf{e} ,
\end{align}
where $D$ is a positive diagonal matrix. $V(\mathbf{e}) > 0$ $\forall$ $\mathbf{e}(t)\neq 0$ and
$V(\mathbf{e}) = 0$ for $\mathbf{e} = 0$. In addition, $V(\mathbf{e})$ is radially unbounded, i.e., $V(\mathbf{e}) \rightarrow \infty , \text{ when } || \mathbf{e} || \rightarrow \infty$. Therefore, if $\dot{V}(\mathbf{e}(t)) < 0$, $\forall$ $\mathbf{e}(t)\neq 0$, then the equilibrium point is globally asymptotically stable.
\begin{align*}
\dot{V}(\mathbf{e}(t)) &= \mathbf{\dot{e}}^{T}(t)A^{T}DA \mathbf{e}(t)+\mathbf{e}^{T}(t)A^{T}DA \mathbf{\dot{e}}(t) \\
&= (- A \mathbf{e}(t))^{T}A^{T}DA \mathbf{e}(t)+ \mathbf{e}^{T}(t)A^{T}DA (- A \mathbf{e}(t)) \\
&= - \mathbf{e}^{T}(t)A^{T}(A^{T}D+DA)A \mathbf{e}(t)
\end{align*}
From the equation above $\dot{V}(\mathbf{e}(t)) < 0$, $\forall$ $\mathbf{e}(t)\neq 0$, if and only if $A^{T}D+DA \succ 0$. Since $A$ is a nonsingular M-matrix, there exists a positive diagonal matrix P such that the matrix $A^{T}P+PA$ is positive definite \cite{1994:mMatrices} (condition $H_{24}$, $p.136$). As a result, there exists a positive diagonal matrix D for which $A^{T}D+DA \succ 0$ and therefore \eqref{lyapcandidate} is a Lyapunov function that stabilizes the system given by \eqref{FMmatrixforminitial}.
\end{proof}

%
%
\section{Joint Distributed Power and Rate Control}\label{algorithm}

%
%
\subsection{System Feasibility Region}\label{feasibilityRegion}

In this subsection, we will provide some intuition on the \emph{SINR} (and  \emph{capacity}) regions of a network, based on the model described in Section \ref{model}. More specifically, we show that they are non-concave sets with respect to the transmitters' power and we suggest that it is possible to reach regions within the convex hull of these non-concave sets. Note that previous works, such as \cite{2003:genie_seeking}, converted the capacity regions to concave sets by working in the high SINR regime, and hence assuming that the capacity is given by $\log_{2}(\Gamma_i)$, instead of $\log_{2}(1+\Gamma_i)$, where $\Gamma_{i}$ is given by formula \eqref{SINR1}.

\begin{defn}\label{feasibility}
(Feasibility)\cite{2003:Gunnarsson}. A set of target SINRs $\Gamma_{i}$ is said to be feasible with respect to a network, if it is possible to assign transmitter powers $p_{i}\geq0$ so that the requirement in inequality \eqref{SINR_geq} is met for all nodes transmitting in the network. Analogously, the power control problem is said to be feasible under the same conditions. Otherwise, the target SINRs and the power control problem are said to be infeasible.
\end{defn}

\noindent We consider a wireless ad hoc network with $M$ transmitters operating in a single channel. Let $\mathbf{p}$ be the $M$-dimensional power vector with $p_{i}$ the power of transmitter $i$ in the network. Moreover, we include upper bounds on the power level, since nodes operate with small batteries which impose limitations on the transmission power. Let $\mathbf{p_{max}}$ be an M-dimensional vector that contains the maximum power level for each transmitter in the network. Note that $p_{i}$ lies in the closed interval $[0,p_{i,\max}]$. We define $\Pi$ to be the set of all possible powers levels in the network, namely, $\Pi(\mathbf{p})\triangleq\{\mathbf{\mathbf{p}}: \mathbf{0}\leq \mathbf{p}\leq \mathbf{p_{max}}\}$. By definition, $\Pi$ is described geometrically by a rectangle for two transmitters only and an M-orthotope\index{orthotope} for M transmitters. Therefore, the set of powers is always a convex set.


\noindent We denote by ${\Gamma}_{i}$ the SINR of transmitter $i$ given by \eqref{SINR1}, and ${\Gamma}_{i,max}$ its
maximum SINR, defined by
\begin{align}
\Gamma_{i,max}\triangleq\{\Gamma_{i}: p_{i}=p_{i,max} \ , p_{j}=0 \ \forall j \in \mathcal{T}, j\neq i\}.
\end{align}
We denote by $\mathbf{\Gamma}$ the $M$-dimensional vector with $\Gamma_{i}$ the SINR of transmitter $i$. We define the set $\Psi$ to be all
feasible data rates that can be achieved in the network at any instant by simultaneous transmissions, i.e., $\Psi(\mathbf{\Gamma})\triangleq\{\mathbf{\mathbf{\Gamma}}: \mathbf{p}\in \Pi\}$. Since $\Gamma_{i}$'s are not concave functions of the transmitters' power levels, depending on the network geometry and topology, $\Psi$ is not necessarily a convex set. Consequently, we define the set $\Phi$
to be the convex hull of $\Psi$. i.e., $\Phi \triangleq \Co(\Psi)$ .

\noindent According to the definition of $\Gamma_{i}$ and the assumption that there exists thermal noise in the network, the set $\Phi$ is always closed and bounded. Hence, although the SINR formula is generally a non convex function, the convex hull operator transforms the SINR region, $\Phi$, into a convex set. Formally, based on the above definitions, we define the \emph{SINR region} of a wireless network as the convex hull of all the feasible SINRs.

\noindent If $\Phi-\Psi=\emptyset$, then $\Psi$ is a convex set, which means that the system performance cannot be improved further. If, for a chosen $\mathbf{\Gamma}$ there exists a solution to the system for simultaneous transmissions, then $\mathbf{\Gamma} \in \Psi$. However, there exist cases where $\mathbf{\Gamma} \in \Phi-\Psi$. In such cases, a scheme that partitions time or  introduces multiple channels should be introduced.

\noindent By Carath\'{e}odory's theorem \cite{2003:convexAnalysisAndOptimisation}, if there is a point $\mathbf{x} \in \mathbb{R}^{M}$, lying in the convex hull of the set $\Psi$, there exists a subset $\Psi '$ of $\Psi$ consisting of no more than $M+1$ points such that $\mathbf{x}$ lies in the convex hull of $\Psi '$. Therefore, any point $\mathbf{{\Gamma}} \in \mathbb{R}^M$, that lies in $\Phi$, can be achieved with a linear combination of no more than $M+1$ points in $\Psi$. This suggests that, by efficiently partitioning the time between different states, target SINRs can be achieved that would otherwise be infeasible. Also, the solution consists of at most $S+1$ different states, where $S$ is the number of transmitters in the network.

\begin{remark}
The same arguments can be used for the case where the capacity $\mathrm{R}_{i}$ is considered, due to the monotonic relation with SINR (since $\mathrm{R}_{i}=\log_{2}\left(1+ \Gamma_{i} \right)$ - Shannon's capacity). Again, capacities are not concave functions of the transmitters' power levels. They depend on the network geometry and topology, and hence, the \emph{capacity region} of a wireless network is defined as the convex hull of all the feasible capacities in the network.
\end{remark}


\begin{figure*}[tp!]
\begin{center}
\subfigure[A wireless ad-hoc network is depicted, consisting of $n=4$ nodes and hence two communication pairs.]
{\includegraphics[width=3in]{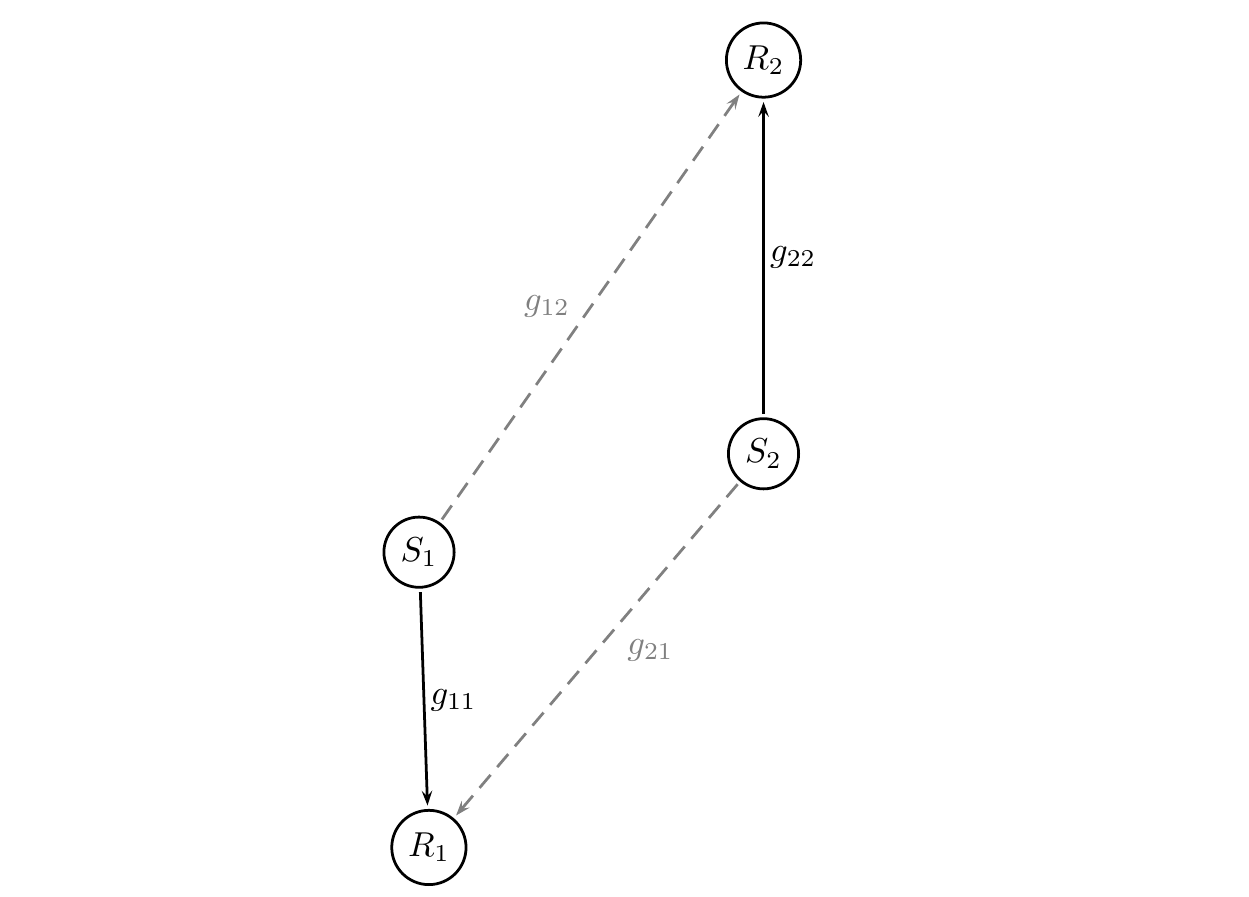}\label{example1a}}
\hfill
\subfigure[In region $K$, simultaneous transmissions can take place $(\mathbf{\mathcal{C}}\in \Psi)$, whereas in region $L$, a higher throughput can be achieved by partitioning the time between different states of data rates $(\mathbf{\mathcal{C}} \in \Phi-\Psi))$. In region $M$, there is no feasible throughput for the network considered $(\mathbf{\mathcal{C}} \notin \Phi)$.]
{\includegraphics[width=3in]{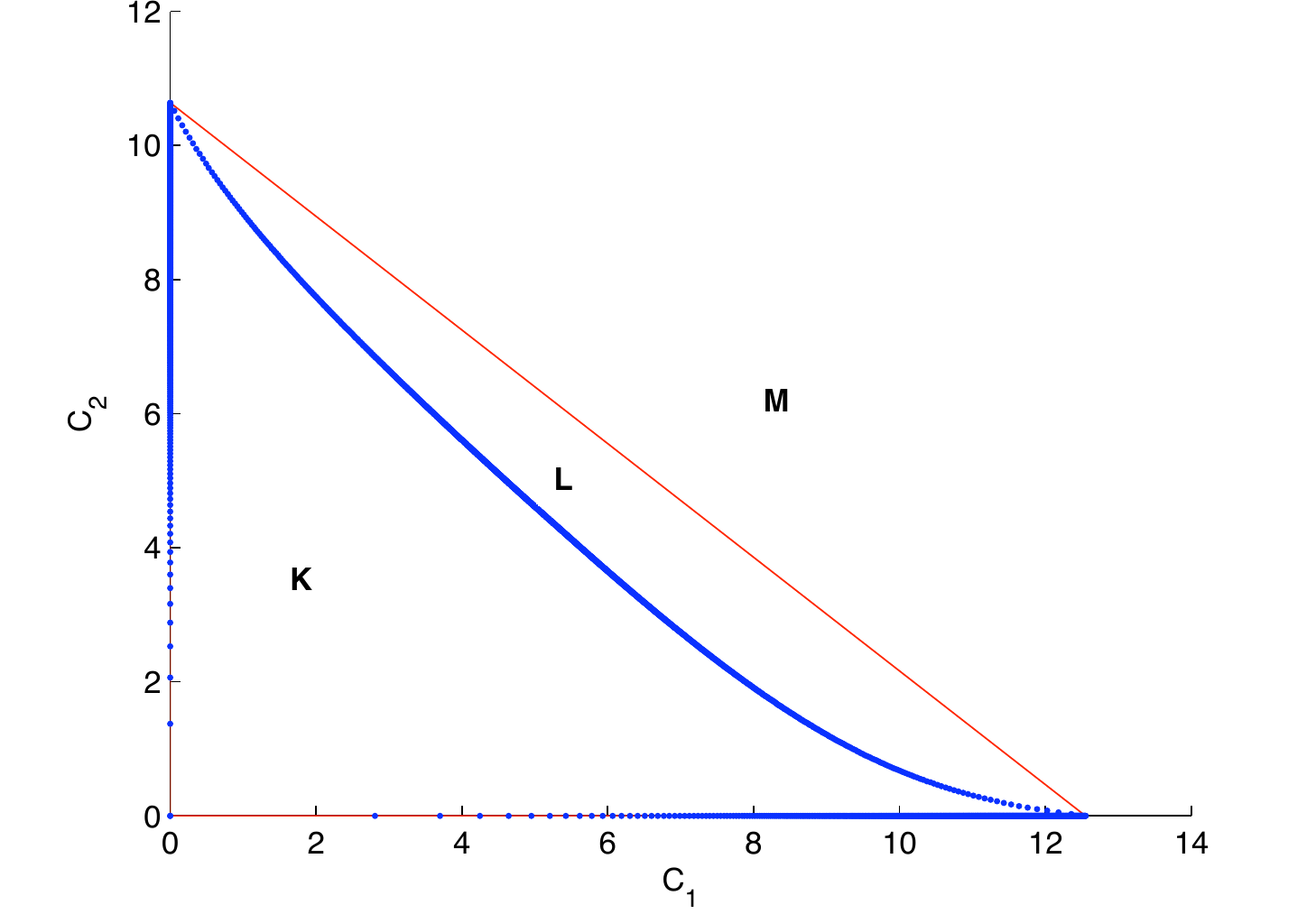}\label{example1b}}
\subfigure[A wireless ad-hoc network is depicted, consisting also of $n=4$ nodes and hence two communication pairs.]
{\includegraphics[width=3in]{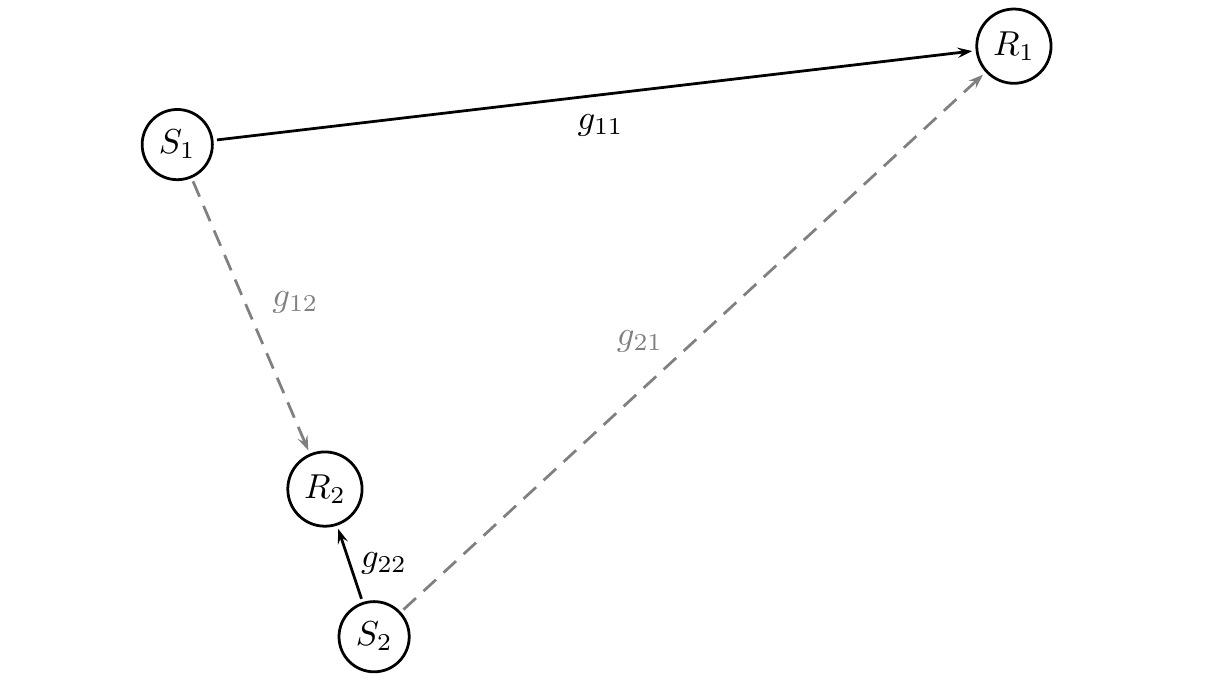}\label{example2a}}
\hfill
\subfigure[The sets of feasible data rates ($\Psi \subset \Phi$) for this network are illustrated. In order to achieve the set of data rates in region $L$, one of the transmitters switches to a different mode of switching between states, whereas the other keeps transmitting at its maximum power.]
{\includegraphics[width=3in]{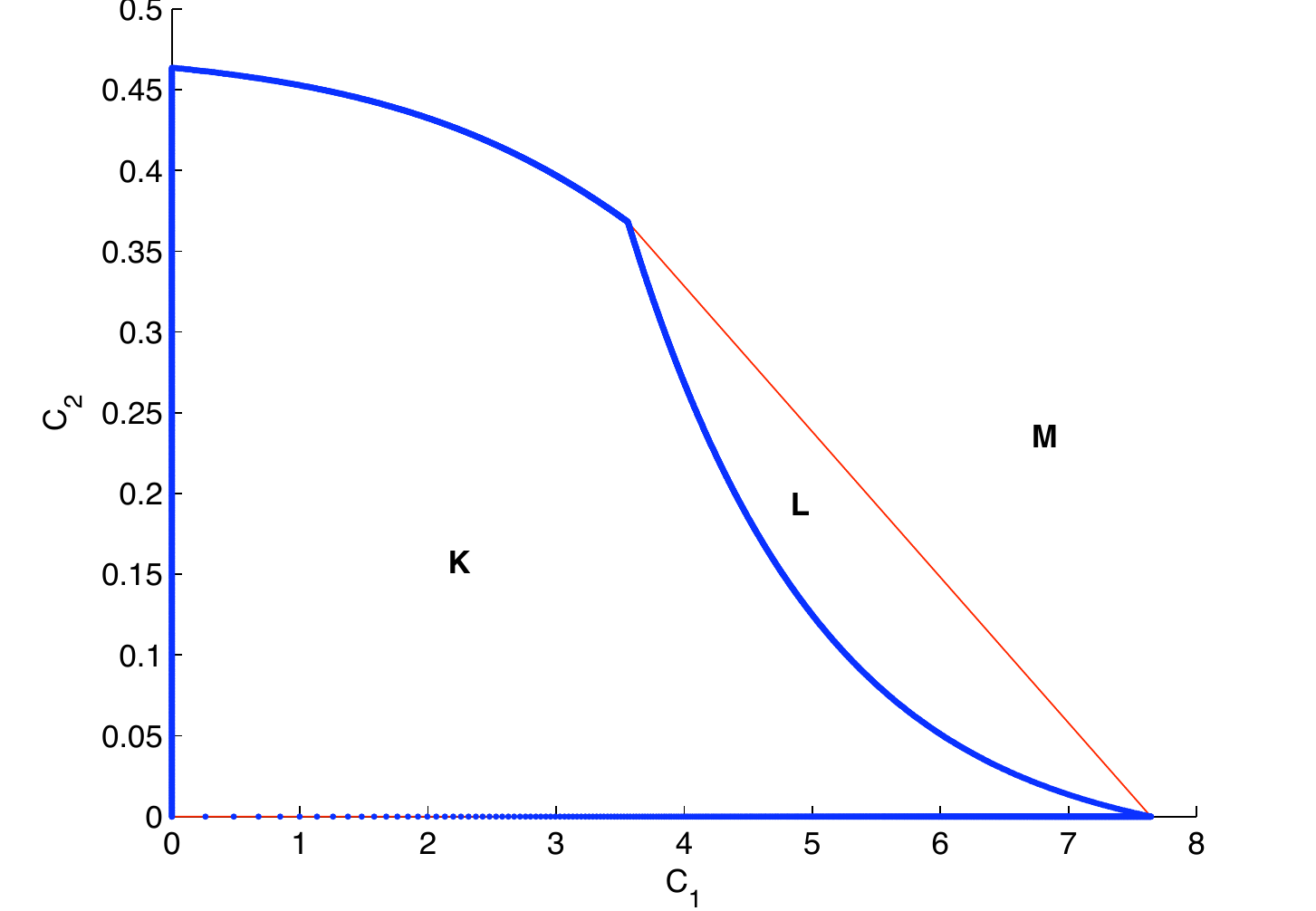}\label{example2b}}
\subfigure[A wireless ad-hoc network is depicted, consisting of $n=6$ nodes and hence two communication pairs.]
{\includegraphics[width=3in]{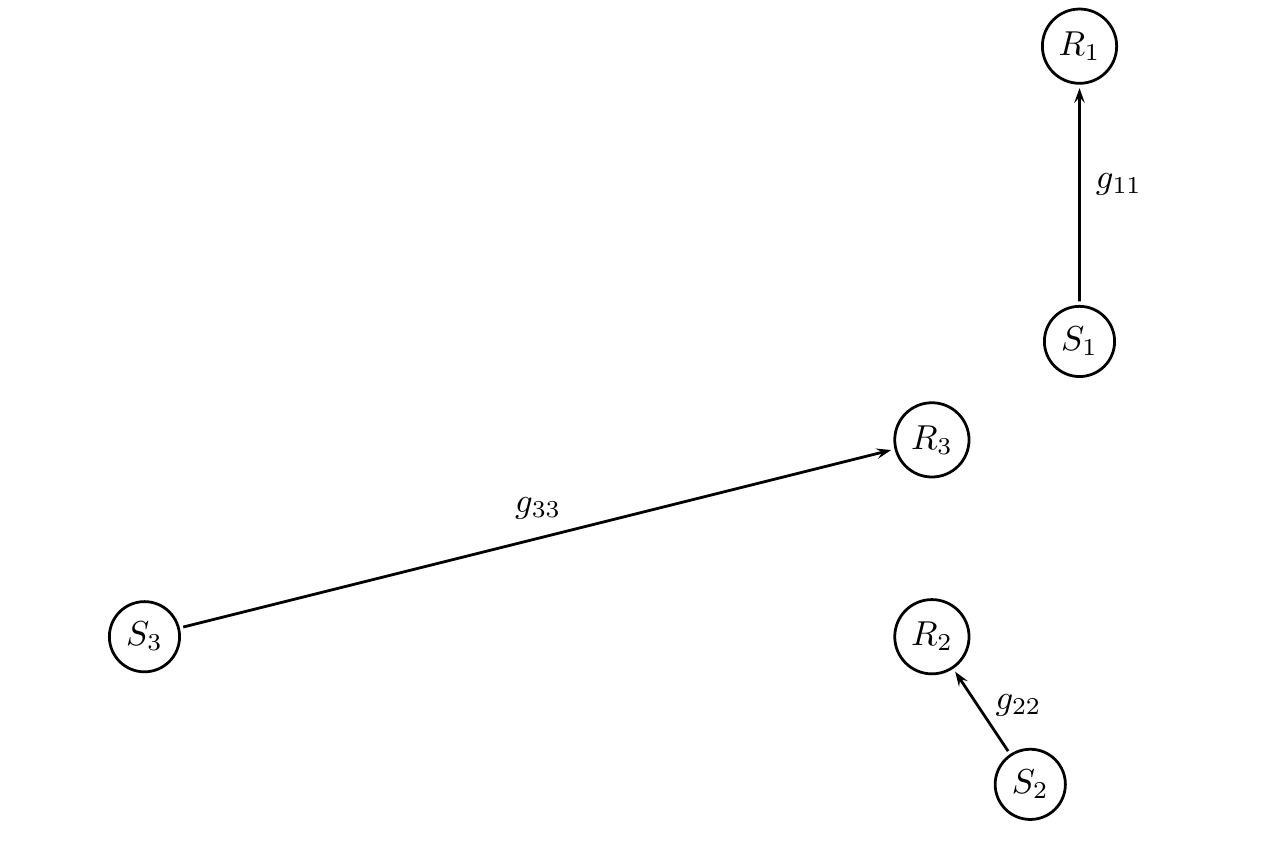}\label{example3a}}
\hfill
\subfigure[The two pairs for which the receiver-transmitter distance is small can transmit simultaneously and $\Phi-\Psi=\emptyset$, whereas the third, due to the large distance between the transmitter and receiver has a low throughput and it deteriorates when any of the other pairs transmits simultaneously.]
{\includegraphics[width=3in]{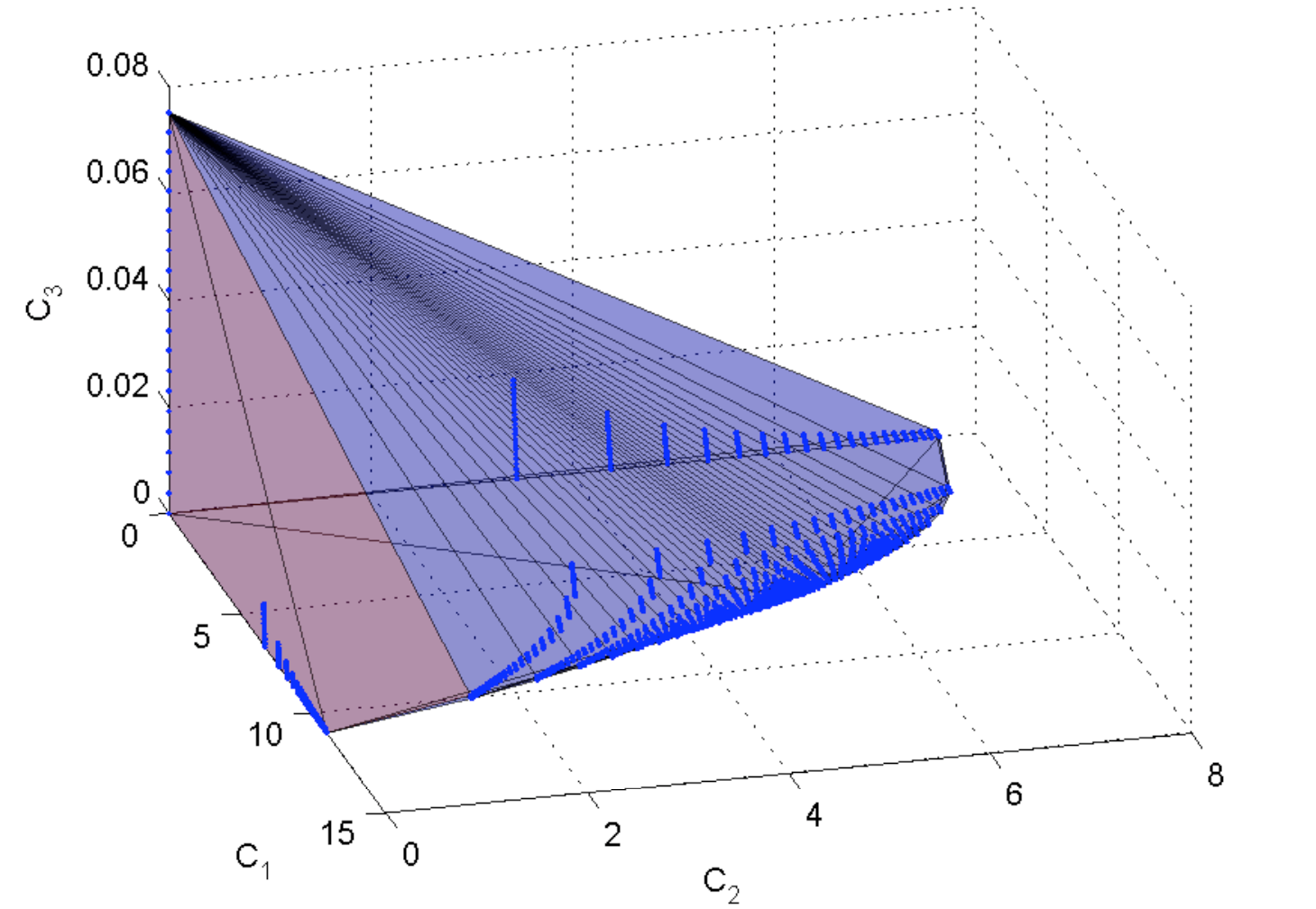}\label{example3b}}
\caption{On the \emph{left-hand side} the wireless ad-hoc network is depicted consisting of the communication pairs. The transmitters appear as nodes with an $S$ and the receivers appear as nodes with an $R$. Communication pairs have the same index and are connected by a solid line \{$S_{i}\rightarrow R_{i}$\}. The grey dotted arrows denote the interference invoked. On the \emph{right-hand side}, the feasible data rates of the communication pairs along with the capacity region of the network are illustrated.}
\label{example1}
\end{center}
\end{figure*}


\noindent Some representative examples of networks are presented in Figure \ref{example1} in order to show the ideas aforementioned. Particularly, we show that there exist networks for which it is possible to attain data rates that would be otherwise impossible, with time division or simultaneous transmissions over many channels. In the following examples (Figure \ref{example1}), on the \emph{left-hand side} the wireless ad-hoc network is depicted consisting of the communication pairs. Each of these pairs consists of a transmitter and a receiver. The transmitters appear as nodes with an $S$ and the receivers appear as nodes with an $R$. Communication pairs have the same index and are connected by a solid line \{$S_{i}\rightarrow R_{i}$\}. The grey dotted arrows denote the interference invoked and it is shown to some of the connections for illustration purposes only. On the \emph{right-hand side}, the feasible data rates of the communication pairs along with the capacity region of the network are illustrated.

%
%
\subsection{Problem formulation}\label{formulation}

\noindent We now formulate the problem of distributed power control in the case which we have more than one channels available and the average SINR of the system is required to meet certain QoS specifications. Therefore, we have a variable SINR threshold in each channel, while the average of all the SINR targets should meet the overall required QoS. 

Let  $i\in \mathcal{T}$ in channel $k\in \mathcal{C}$, where $\mathcal{T}$ and $\mathcal{C}$ are the sets of wireless transmitters and available channels, respectively.  As before, $\gamma_{i}$ denotes the desired average SINR of user $i$ for all $N$ channels. We introduce $x_{i,k}$ to be the allotted (or desired) SINR target for user $i$ and $w_{i,k}$ is the \emph{effective} interference at receiver $i$ in channel $k$, and is given by
\begin{align}\label{wik}
w_{i,k}=\sum_{j\in \mathcal{T}_{-i}}\frac{g_{ji}}{g_{ii}}p_{j,k}+ \frac{\nu_i}{g_{ii}}.
\end{align}
Therefore, for each channel $k$ we require that the SINR is greater than or equal to $x_{i,k}$:
\begin{align}
\frac{p_{i,k}}{w_{i,k}} \geq x_{i,k} ,\quad	x_{i,k} \geq 0
\end{align}
where $p_{i,k}$ is the power of user $i$ in channel $k$ The target SINR ($\gamma_i$) should be met on average over all $N$ channels, i.e.,
\begin{align}
\frac{1}{N}\sum_{k\in \mathcal{C}}x_{i,k} \geq \gamma_{i} .
\end{align}
Therefore, we are required to develop an algorithm that adapts the targeted SINR in each channel, accounting for the individual channel's conditions, in addition to updating the power.

%
%
%
%
\subsection{Main results} \label{results}

\noindent We propose a distributed algorithm that not only updates the power level $p_{i,k}$ of each node $i$ in a certain channel $k$, but also updates the desired SINR $x_{i,k}$ for each channel in order to take advantage of the channel conditions. We show that this algorithm is able to find feasible solutions that would be inaccessible via conventional power control algorithms (e.g., the FM algorithm) that operate in a single channel only. The proposed algorithm aims to find an equilibrium for the system consisting of all the communication pairs making use of all the available channels, if a feasible solution exists.

\begin{defn}\label{def:feasibleMultiple}
A feasible solution exists if there exist SINR targets
$\mathbf{\hat{x}}_{k}=[\begin{array}{cccc}
\hat{x}_{1,k} & \hat{x}_{2,k} & \ldots & \hat{x}_{N,k}
\end{array}]^{T}$ such that
\begin{align}
\rho\left(C(\mathbf{\hat{x}}_{k})\right)<1, \ ~\forall~k\in \mathcal{C}
\end{align}
and
\begin{align}
\frac{1}{N}\sum_{k\in \mathcal{C}}\hat{x}_{i,k}=\gamma_{i} , \ ~\forall~i\in \mathcal{T} .
\end{align}
\end{defn}

\noindent In the following theorem we state that the proposed distributed algorithm is asymptotically stable and converges as long as the allotted SINRs belong to a feasible solution. 
\begin{theorem}\label{th:1}
Assume that  $\gamma_{i}$, $g_{ji}$, $\nu_i \in \mathds{R}_{+}$, for any initial state $p_{i}(0)>0$ and for any proportionality constants $c_{i,k},b_{i,k},\zeta_{i} \in \mathds{R}_{+}$, and suppose the spectral radius of matrix $C(\mathbf{x}_{k}(t))~\forall~ t$ in \eqref{C} remains less than $1$. Then, the distributed power control formula 
\begin{align}\label{pdot_formula}
\dot{p}_{i,k}(t)=c_{i,k} \left(-p_{i,k}(t)+x_{i,k}(t)\left({\sum_{j\neq i, j \in \mathcal{T}}{\frac{g_{ji}}{g_{ii}}p_{j,k}(t) + \frac{\nu_{i,k}}{g_{ii}}}}\right)\right),
\end{align}
and the distributed allotted SINR target update formula
\begin{align}\label{xdot_formula}
\dot{x}_{i,k}(t) = b_{i,k}(t) \left[\frac{\zeta_{i}}{N}\left( \gamma_{i}- \frac{1}{N} \sum_{k\in \mathcal{C}}x_{i,k}(t) \right)- w_{i,k}(t)\left(x_{i,k}(t)w_{i,k}(t)-p_{i,k}(t) \right)\right] \end{align}
converge to a feasible solution.
\end{theorem}

\begin{proof}
For each transmitter $i\in \mathcal{T}$, we let the utility function be
\begin{align}\label{Ui_single}
U_{i}:=d_{i}\zeta_{i} \left( \gamma_{i}- \frac{1}{N} \sum_{k\in \mathcal{C}}x_{i,k} \right)^{2}+ \sum_{k\in \mathcal{C}}d_{i}\left[c_{i,k}\left(x_{i,k}w_{i,k}-p_{i,k} \right)\right]^{2}
\end{align}
where $c_{i,k}$, $d_{i}$ and $\zeta_{i}$ are positive constants. 

Define $\mathbf{e}_{k} \triangleq \mathbf{p}_{k}-\mathbf{p}_{k}^{*}$, $\mathbf{p}_{k} \triangleq \left(
                    \begin{array}{cccc}
                      p_{1}^{k} & p_{2}^{k} & \ldots & p_{M}^{k} \\
                    \end{array}
                  \right)^{T}$, $\mathbf{x}_{k} \triangleq \left(
                    \begin{array}{cccc}
                      x_{1}^{k} & x_{2}^{k} & \ldots & x_{M}^{k} \\
                    \end{array}
                  \right)^{T}$, $D\triangleq \diag(d_{i})$ and $A(\mathbf{x}_{k})$ is the same as before, but with $x_{i,k}$ in place of $\gamma_{i}$ for each node respectively. Explicitly, since A=KH, we express $H$ as
\begin{align}
H_{ij}=\begin{cases}
1 & \text{, if $i=j$},\\
-x_{i,k}\frac{g_{ji}}{g_{ii}} & \text{, if $i \neq j$}.
\end{cases}
\end{align}

Therefore the utility function for the whole network, with $M$ communication pairs, is given by,
\begin{align*}
U = \sum_{i=1}^{M}U_{i} &= \sum_{i=1}^{M}\left[ d_{i}\zeta_{i} \left( \gamma_{i}- \frac{1}{N} \sum_{k \in \mathcal{C}}x_{i,k} \right)^{2}+ \sum_{k\in \mathcal{C}}d_{i}\left[c_{i,k}\left(x_{i,k}w_{i,k}-p_{i,k} \right)\right]^{2} \right] \\
&= \sum_{i=1}^{M}d_{i}\zeta_{i} \left( \gamma_{i}- \frac{1}{N} \sum_{k \in \mathcal{C}}x_{i,k} \right)^{2}+\sum_{i=1}^{M}\sum_{k \in \mathcal{C}}d_{i}\left[c_{i,k}\left(x_{i,k}w_{i,k}-p_{i,k} \right)\right]^{2} \\
&= \sum_{i=1}^{M}d_{i}\zeta_{i} \left( \gamma_{i}- \frac{1}{N} \sum_{k \in \mathcal{C}}x_{i,k} \right)^{2}+\sum_{k \in \mathcal{C}}\sum_{i=1}^{M}d_{i}\left[c_{i,k}\left(x_{i,k}w_{i,k}-p_{i,k} \right)\right]^{2} \\
&=\underbrace{\sum_{i=1}^{M}d_{i}\zeta_{i} \left( \gamma_{i}- \frac{1}{N} \sum_{k \in \mathcal{C}}x_{i,k} \right)^{2} }_{U_{1}} +\underbrace{\sum_{k\in \mathcal{C}}\left(\mathbf{e}_{k}^{T}A^{T}(\mathbf{x}_{k})DA(\mathbf{x}_{k})\mathbf{e}_{k}\right)}_{U_{2}} \label{Ulast}
\end{align*}

\noindent By definition, the utility function is nonnegative, i.e., 
\begin{align}
U > 0, ~~\forall ~\mathbf{e}_{k} \neq 0~\mbox{and}~ \frac{1}{N} \sum_{k\in \mathcal{C}}x_{i,k} \neq \gamma_{i}~ \forall i.
\end{align}  
Also, 
\begin{align}
U = 0~ \mbox{for}~ \mathbf{e}_{k} =0~\mbox{and}~ \frac{1}{N} \sum_{k\in \mathcal{C}}x_{i,k} = \gamma_{i}~ \forall i.
\end{align}  
In addition, $U$ is radially unbounded, i.e., $U \rightarrow \infty$, when $|| \mathbf{e}_{k} || \rightarrow \infty$ or/and $|| \mathbf{x}_{k} || \rightarrow \infty$. 

\noindent In what follows, we find the conditions for which $\dot{U} < 0$, $\forall$ $\mathbf{e}_{k} \neq 0$ and $\sum_{k\in \mathcal{C}}x_{i,k}/N \neq \gamma_{i}$. That is,
\begin{align}
\frac{dU}{dt} = \sum_{i=1}^{M}\sum_{k\in \mathcal{C}} \frac{\partial U}{\partial x_{i,k}}\frac{d  x_{i,k}}{dt} + \sum_{k\in \mathcal{C}}\frac{\partial U}{\partial \mathbf{e}_{k}}\frac{d  \mathbf{e}_{k}}{dt} < 0 .
\end{align}
for all $\mathbf{e}_{k}\neq 0$ and $\sum_{k}x_{i,k}/N \neq \gamma_{i}$. In Lemma \ref{chap5_lemma}, we showed that the Lyapunov function \eqref{lyapfunction} that proves stability for the FM algorithm is of the same form as $U_{2}^{k}$; namely, $U_{2}^{k}=\mathbf{e}_{k}^{T}A^{T}(\mathbf{x}_{k})DA(\mathbf{x}_{k})\mathbf{e}_{k}$. Convergence in case the matrix A is time varying, $A(\mathbf{x}_{k}(t))$, is proven in \cite{2005:BambosIT,2007:olama}. Therefore, there exists positive diagonal matrix $D$ such that $A^{T}D+DA>-Q$ for any positive definite matrix $Q$, and given that $\rho(C(\mathbf{x}_{k}(t)))<1$, such that,
\begin{align}
\sum_{k\in \mathcal{C}}\frac{\partial U}{\partial \mathbf{e}_{k}}\frac{d  \mathbf{e}_{k}}{dt} < 0 , \ \ \ \ \ \forall \ \  \mathbf{e}_{k} \neq 0 .
\end{align}
$U_{1}^k$ is independent of $ \mathbf{e}_{k}(t)$ and therefore $\dot{U}< 0 ,~\forall ~\mathbf{e}_{k}(t) \neq 0$. Thus, the power update formula \eqref{pdot_formula} stabilizes the system. Now, we show that the update formula for $x_{i,k}$ establishes that
\begin{align}
\sum_{i=1}^{M}\sum_{k\in \mathcal{C}} \frac{\partial U}{\partial x_{i,k}}\frac{d  x_{i,k}}{dt} < 0 , \ \ \ \ \ \ \frac{1}{N} \sum_{k\in \mathcal{C}}x_{i,k} \neq \gamma_{i}~ \forall i.
\end{align}

\noindent For a single transmitter,
\begin{align}
\frac{\partial U_{i}}{\partial x_{i,k}}= &-\frac{2}{N}d_{i}\zeta_{i}\left( \gamma_{i}- \frac{1}{N} \sum_{k\in \mathcal{C}}x_{i,k} \right) +2d_{i}w_{i,k}\left(x_{i,k}w_{i,k}-p_{i,k} \right) .
\end{align}
\begin{align}
\sum_{k\in \mathcal{C}}\left(\frac{\partial U_{i}}{\partial x_{i,k}}\dot{x}_{i,k}\right) &= \sum_{k\in \mathcal{C}} \left[-\frac{2}{N}d_{i}\zeta_{i}\left( \gamma_{i}- \frac{1}{N} \sum_{k\in \mathcal{C}}x_{i,k} \right)+2d_{i}w_{i,k}\left(x_{i,k}w_{i,k}-p_{i,k} \right)\right]\dot{x}_{i,k} \\
&= - 2d_{i} \sum_{k\in \mathcal{C}} \left[\frac{\zeta_{i}}{N}\left( \gamma_{i}- \frac{1}{N} \sum_{k\in \mathcal{C}}x_{i,k} \right)-w_{i,k}\left(x_{i,k}w_{i,k}-p_{i,k} \right)\right]\dot{x}_{i,k} \label{5.48}\\
&= -2d_{i} \sum_{k\in \mathcal{C}}b_{i,k}\left[\frac{\zeta_{i}}{N}\left( \gamma_{i}- \frac{1}{N} \sum_{k\in \mathcal{C}}x_{i,k} \right)-w_{i,k}\left(x_{i,k}w_{i,k}-p_{i,k} \right)\right]^{2}<0, \label{5.49}
\end{align}
\noindent where \eqref{5.49} is obtained by substituting \eqref{xdot_formula} into \eqref{5.48}. Since $x_{i,k}$ does not appear in the other nodes' utility functions, it is shown that for each user, the allotted SINR target's differential equation makes the rate of change of the utility function of node $i$ with respect to $x_{i,k}$ negative.
\end{proof}
\vspace{0.5cm}
\noindent We have proved the update formulae converge to a feasible solution under the condition $\rho(C(\mathbf{x}_{k}))<1$. Hence, the suggested algorithm converges to a feasible solution as long as the desired SINRs belong to the feasible set of SINRs. 

\noindent However, this might not be always the case, and $\rho(C(\mathbf{x}_{k}))>1$ for some time-interval. In what follows, we introduce a condition on the update formulae such that the algorithm converges to a feasible solution, if one exists, even if $\rho(C(\mathbf{x}_{k}))>1$ over bounded time intervals.

\begin{theorem}
Define
\begin{align*}
\theta_{i}(t):=\gamma_{i}-\frac{1}{N}\sum_{k\in \mathcal{C}}x_{i,k}(t) ,
\end{align*}
where $N$ is the number of available channels, $x_{i,k}$ is the allotted SINR target for transmitter $i\in\mathcal{T}$ in channel $k\in\mathcal{C}$ and $\gamma_{i}$ is the average capture ratio desired at receiver $i$. The distributed power control formula 
\begin{align}\label{pdot_formula2}
\dot{p}_{i,k}(t)=\begin{cases}
0 & \mbox{, if } \theta_i\neq0 ~\mbox{and}~\dot{x}_{i,k}=0, ~\forall k\in\mathcal{C}, \\
\displaystyle c_{i,k} \left(-p_{i,k}(t)+x_{i,k}(t)\left({\sum_{j\neq i, j \in \mathcal{T}}{\frac{g_{ji}}{g_{ii}}p_{j,k}(t) + \frac{\nu_{i,k}}{g_{ii}}}}\right)\right) & \mbox{, otherwise}.
\end{cases}
\end{align}
and the distributed allotted SINR target update formula
\begin{align}\label{xdot_formula1}
\dot{x}_{i,k}(t) = b_{i,k}(t) \left[\frac{\zeta_{i}}{N}\left( \gamma_{i}- \frac{1}{N} \sum_{k\in \mathcal{C}}x_{i,k}(t) \right)- w_{i,k}(t)\left(x_{i,k}(t)w_{i,k}(t)-p_{i,k}(t) \right)\right] 
\end{align}
converge to a feasible solution, if one exists, for all positive constants $c_{i,k}\in \mathds{R}_{+}$,  $\zeta_{i} \in \mathds{R}_{+}$, and provided the proportionality gain $b_{i,k}(t)\in \mathbb{R}_{+}$ of the allotted SINR target update formula \eqref{xdot_formula1} is appropriately chosen such that the inequality
\begin{align}\label{condition2_chap5}
b_{i,k}(t)\left[\frac{\zeta_{i}}{N}\left( \gamma_{i}- \frac{1}{N} \sum_{k\in \mathcal{C}}x_{i,k} \right)-w_{i,k}\left(x_{i,k}w_{i,k}-p_{i,k} \right)\right]^2 \geq 2\dot{p}_{i,k}(t) {\ddot{p}}_{i,k}^{2}(t)
\end{align}
is satisfied by all nodes $i$ in at least one of the channels $k\in\mathcal{C}$ and 
\begin{align}\label{U_idot1}
\dot{U_i} =d_{i}\sum_{k\in \mathcal{C}}\left( -\frac{\dot{x}^{2}_{i,k}}{b_{i,k}}+2\phi_{i,k}\dot{\phi}^{2}_{i,k} \right)<0.
\end{align}.
\end{theorem}

\begin{proof}
For each transmitter $i$, we let the utility function \eqref{Ui_single} as before. The global utility function is given by the summation of all the individual utility functions and can be written as
\begin{align}\label{ui_new}
U=\sum_{i=1}^{M} U_{i} = \sum_{i=1}^{M}d_{i} \left[\zeta_{i}\left( \gamma_{i}- \frac{1}{N} \sum_{k\in \mathcal{C}}x_{i,k} \right)^2+\sum_{k}\dot{p}^{2}_{i,k} \right] .
\end{align}
Let $\phi_{i,k}=\dot{p}_{i,k}$, then $\dot{\phi}_{i,k}=\ddot{p}_{i,k}$. Hence,
\begin{align}\label{dotU}
\dot{U} =\sum_{i=1}^{M} \dot{U}_{i} = \sum_{i=1}^{M}\left[ \left(\sum_{k\in \mathcal{C}} \frac{\partial U_{i}}{\partial x_{i,k}} \dot{x}_{i,k} \right) + \left(\sum_{k\in \mathcal{C}} \frac{\partial U_{i}}{\partial \phi_{i,k}} \dot{\phi}_{i,k} \right)\right].
\end{align}
By substituting $U_{i}$ as given in \eqref{ui_new} into \eqref{dotU},
\begin{align}\label{Udot}
\dot{U} = \sum_{i=1}^{M}d_{i}\sum_{k\in \mathcal{C}}\left( -\frac{\dot{x}^{2}_{i,k}}{b_{i,k}}+2\phi_{i,k}\dot{\phi}^{2}_{i,k} \right).
\end{align}
This implies that $\dot{U}\leq 0$ if 
\begin{align}\label{eq:cond2}
-\frac{\dot{x}^{2}_{i,k}}{b_{i,k}}+2\phi_{i,k}\dot{\phi}^{2}_{i,k} \leq 0, ~\forall~k \in \mathcal{C}~ \mbox{and} ~\forall~i \in \mathcal{T}. 
\end{align}
Therefore, by substituting \eqref{xdot_formula1} into \eqref{eq:cond2}, the following condition is derived:
\begin{align}\label{condition_chap5}
b_{i,k}(t)\left[\frac{\zeta_{i}}{N}\left( \gamma_{i}- \frac{1}{N} \sum_{k\in \mathcal{C}}x_{i,k} \right)-w_{i,k}\left(x_{i,k}w_{i,k}-p_{i,k} \right)\right]^2
\geq 2\dot{p}_{i,k}(t) {\ddot{p}}_{i,k}^2(t).
\end{align}

\noindent For $\dot{x}_{i,k}\neq 0$, if $\dot{p}_{i,k}(t)<0$, then since $\ddot{p}_{i,k}^{2}(t)>0$, inequality \eqref{condition_chap5} holds for any proportionality gain $b_{i,k}(t)>0$, whereas if $\dot{p}_{i,k}(t)>0$, then $b_{i,k}(t)$ can be adapted so that \eqref{condition_chap5} holds. 

\noindent For $\dot{x}_{i,k}= 0$ and if $\dot{p}_{i,k}(t)>0$, then the inequality is not fulfilled in channel $k$. However, as long there exists at least 
a channel, say $m$, which satisfies $\dot{x}_{i,m}\neq 0$, then the proportionality gain $b_{i,m}(t)$ can be adjusted such that
\begin{align}\label{U_idot}
\dot{U_i} =d_{i}\sum_{k\in \mathcal{C}}\left( -\frac{\dot{x}^{2}_{i,k}}{b_{i,k}}+2\phi_{i,k}\dot{\phi}^{2}_{i,k} \right)<0.
\end{align}
while fulfilling inequality \eqref{condition_chap5} as well. 

\noindent For $\dot{x}_{i,k}= 0~\forall k\in \mathcal{C}$ and $\theta_i \neq 0$, then $\dot{p}_{i,k}(t)=0$. Hence \eqref{xdot_formula1} becomes
\begin{align}\label{eq:xdotpdot0}
\dot{x}_{i,k}(t) = b_{i,k}(t) \left[\frac{\zeta_{i}}{N}\left( \gamma_{i}- \frac{1}{N} \sum_{k\in \mathcal{C}}x_{i,k}(t) \right)\right] \mbox{ and } \dot{x}_{i,k}(t) \neq 0. 
\end{align}

\noindent For  $\theta_i = 0$, then $\dot{x}_{i,k}\neq 0~\forall k\in \mathcal{C}$ unless $\dot{p}_{i,k}(t)=0$ for which equilibrium is reached. Hence, the condition on the power update formula establishes that $\dot{U} < 0$ in all cases, apart from the case in which a feasible solution is reached.
\end{proof}

\begin{remark}
If the maximum power $p_{i,\max}$ is reached, the algorithm continuous to update both formulae as before, and if $\dot{p}_{i,k}(t)<0$, then the power is updated. Hence, the upper bound of the power does not affect the analysis of this work.
\end{remark}


\section{Illustrative Examples} \label{examples}

To demonstrate the essence of the results derived, we present some simple examples. We consider a simple wireless ad-hoc network consisting of $n=4$ nodes (i.e. two communication pairs $S_{i}\rightarrow R_{i}$). The network structure is shown in Figure \ref{example}. The network is symmetric, so initial powers and/or desired SINRs should be different for the pairs in order to observe differentiation in their actions.

\begin{figure}[h]
 \centering
  \includegraphics[width=1.8in]{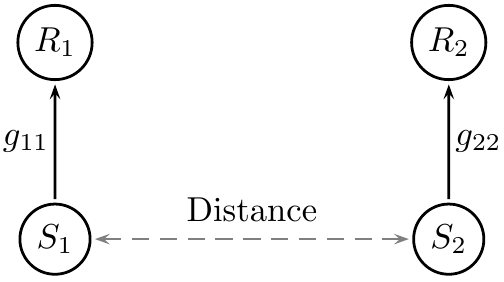}
  \caption{A wireless ad-hoc network of $n=4$ nodes, consisting of two communication pairs \{$S_{i}\rightarrow R_{i}$\}.}
  \label{example}
\end{figure}

By varying the distance between the two communication pairs, the SINR region of the network is changed. Thus, we can either choose a convex SINR region to study or a non-convex one. Based on this network prototype we examine the performance of our algorithm in different cases of network configuration.

The parameters used for the cases studied are summarized below:

\begin{table}[H]
{\small
    \renewcommand{\arraystretch}{1.3}
    \centering
\begin{tabular}{l|r}
  \hline
  Parameter & Value \\
  \hline \hline
  Average SINR target ($\gamma_{i}$) & 3 \\
  Constant ($\zeta_{i}$) & 20 \\
  Proportionality constant ($k_{i}$) & 1 \\
  Initial Proportionality gain ($b_{i,k}$) & 200 \\
  Noise ($\nu_{i,k}$) & 0.04 $\mu W$ \\
  Maximum Power ($p_{i,k}^{max}$) & $10^6$ $\mu W$ \\
  \hline
\end{tabular}
\caption[Parameters of the algorithm and the wireless networks used in the examples]{Parameters of the algorithm and the wireless networks used in the following examples. Power and noise are measured in \emph{Watts} (W) and data rate in bits per second (\emph{bits/s}).}
\label{ParametersVarSINR}}
\end{table}

The initial power levels ($p_{i,k}$) for each node $i$ and for each channel $k$ are chosen at random. The upper bound for the maximum power for the simulations is set to $10^{6}mW$, large enough to allow the distributed algorithms to operate without this extra condition on the power, since we have not considered an upper bound in our derivations. The proportional gain $b_{i,k}$ is initially set to $200$ for all transmitters in the network, which is high enough to guarantee that condition \eqref{condition2_chap5} is fulfilled and it is increased whenever required, depending on the condition so that it maintains a safe margin from the minimum value.

\subsection{Example 1: Feasible network}

In this example, we consider the following network (Figure \ref{ex0}) where the distance between the transmitters (and the receivers) is 8 meters. In this network the interference each receiver experiences is lower compared to the received signal.

\begin{figure}[h]
    \centering
  \includegraphics[width=3.3in]{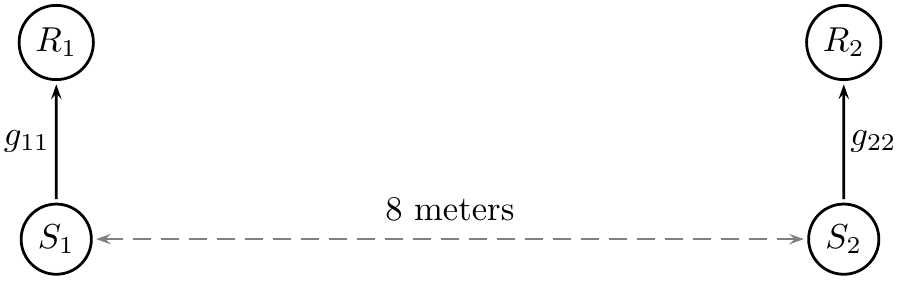}
  \caption{A wireless ad-hoc network of $n=4$ nodes, consisting of two ommunication pairs \{$S_{i}\rightarrow R_{i}$\}. The distance between the transmitters is 8 m.}
  \label{ex0}
\end{figure}

The SINR region is convex and hence the required QoS can be easily achieved by simultaneous transmissions, even for fixed QoS requirements, as shown in Figure \ref{chap5_ex0_sinrRegion}.

\begin{figure}[h]
    \centering
  \includegraphics[width=4in]{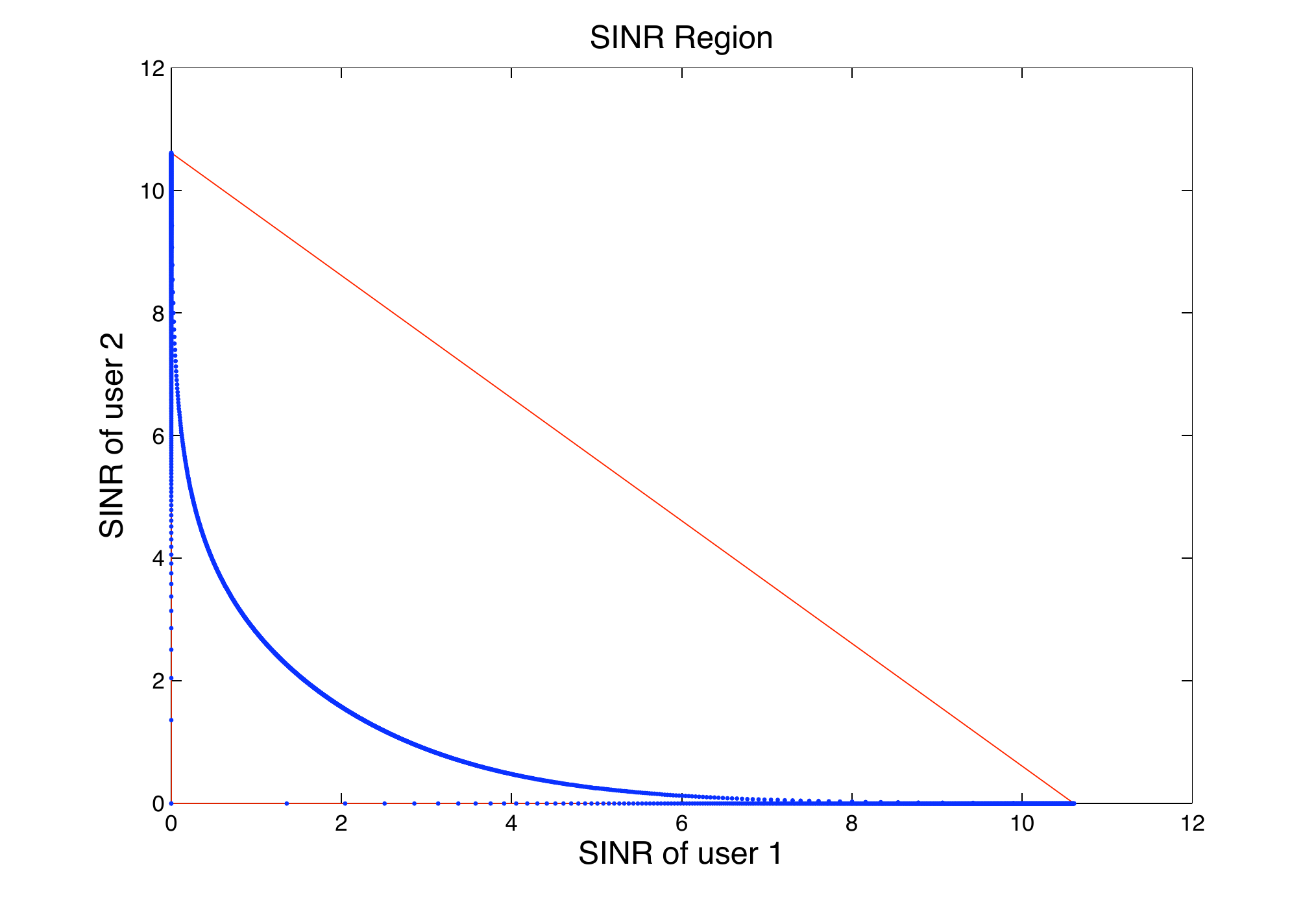}
  \caption{The SINR region of wireless ad-hoc network of Figure \ref{ex0}. It is convex and hence any set of SINRs can be realized with simultaneous transmissions.}
  \label{chap5_ex0_sinrRegion}
\end{figure}

The graphs in Figure \ref{chap5_ex0_all} illustrate the response of our distributed algorithm to the network considered. Each column represents one channel. In the first row, the power levels for all users over time are depicted. In the second row, the allotted SINR targets over time are shown, whereas in the third row the figures show actual SINR values for each of the channels. The graphs show very fast convergence and since the initial powers are very close, the distributed algorithms behave similarly. The system converges to a feasible solution quite fast, even though it is not the optimal one in terms of minimizing the individual's total power. In addition, both transmitters achieve the required QoS on average.

\begin{figure*}[h]
   \centering
    \includegraphics[width=3.2in]{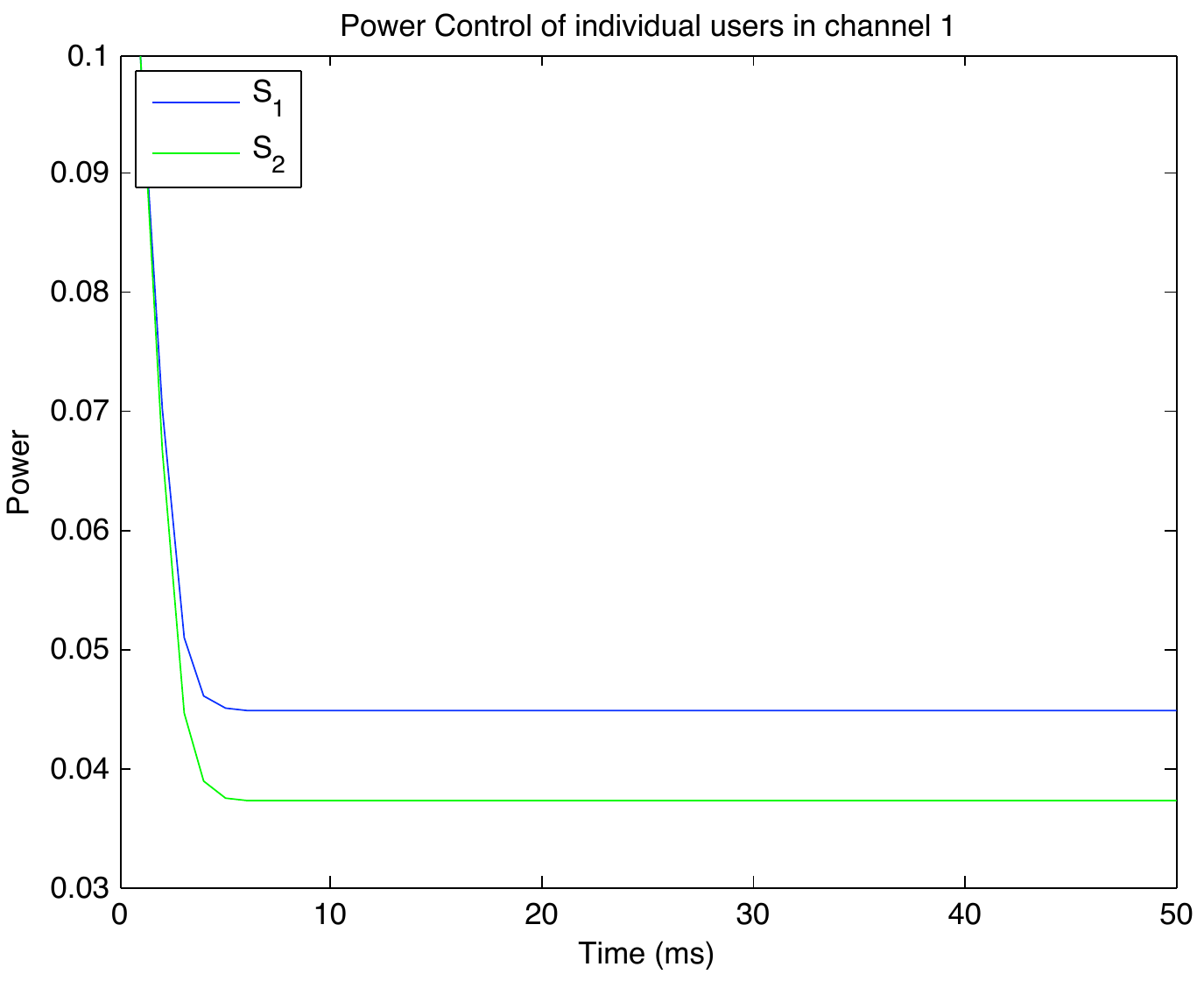}\hfill
    \includegraphics[width=3.2in]{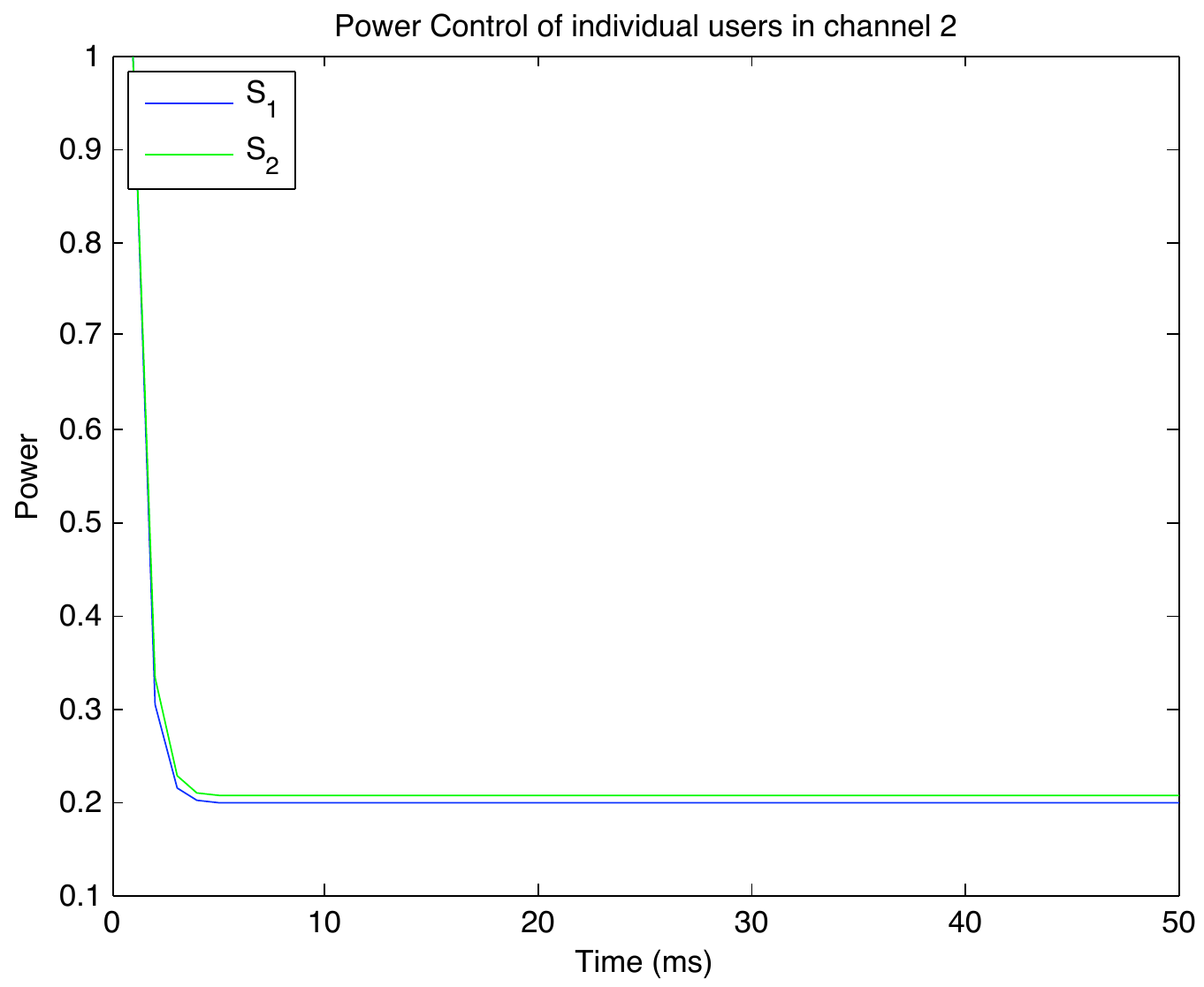}
    \includegraphics[width=3.2in]{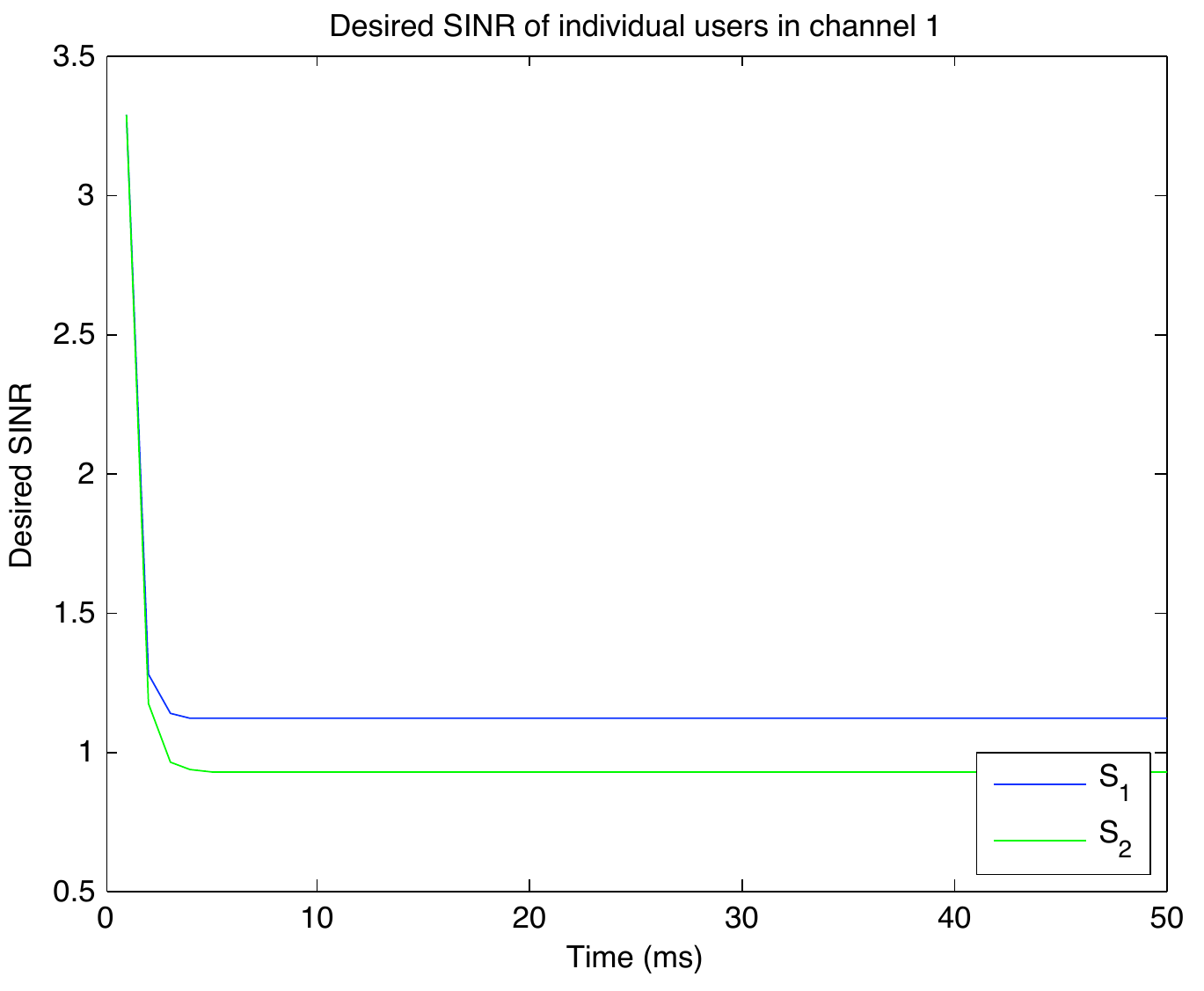}\hfill
    \includegraphics[width=3.2in]{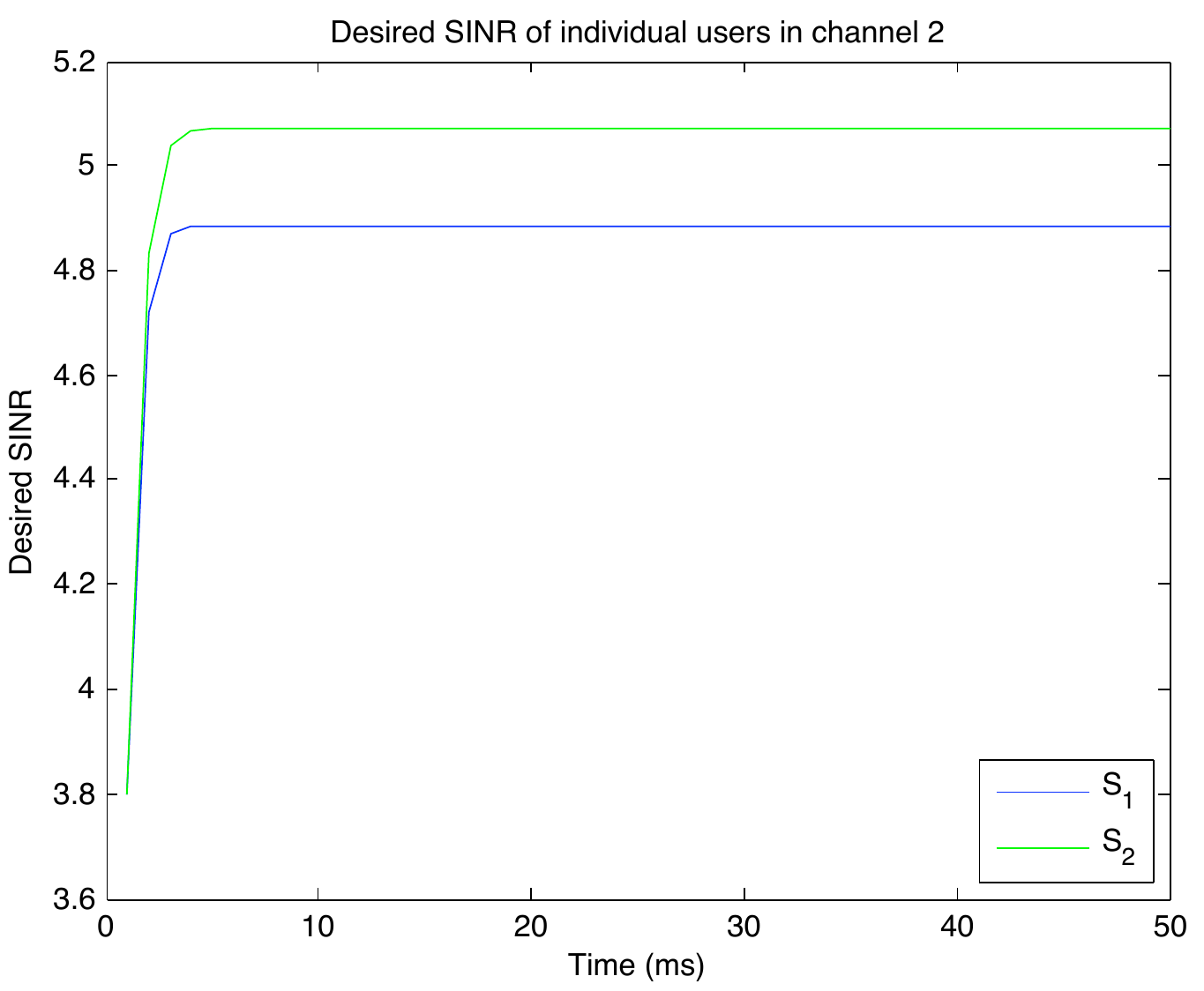}
    \includegraphics[width=3.2in]{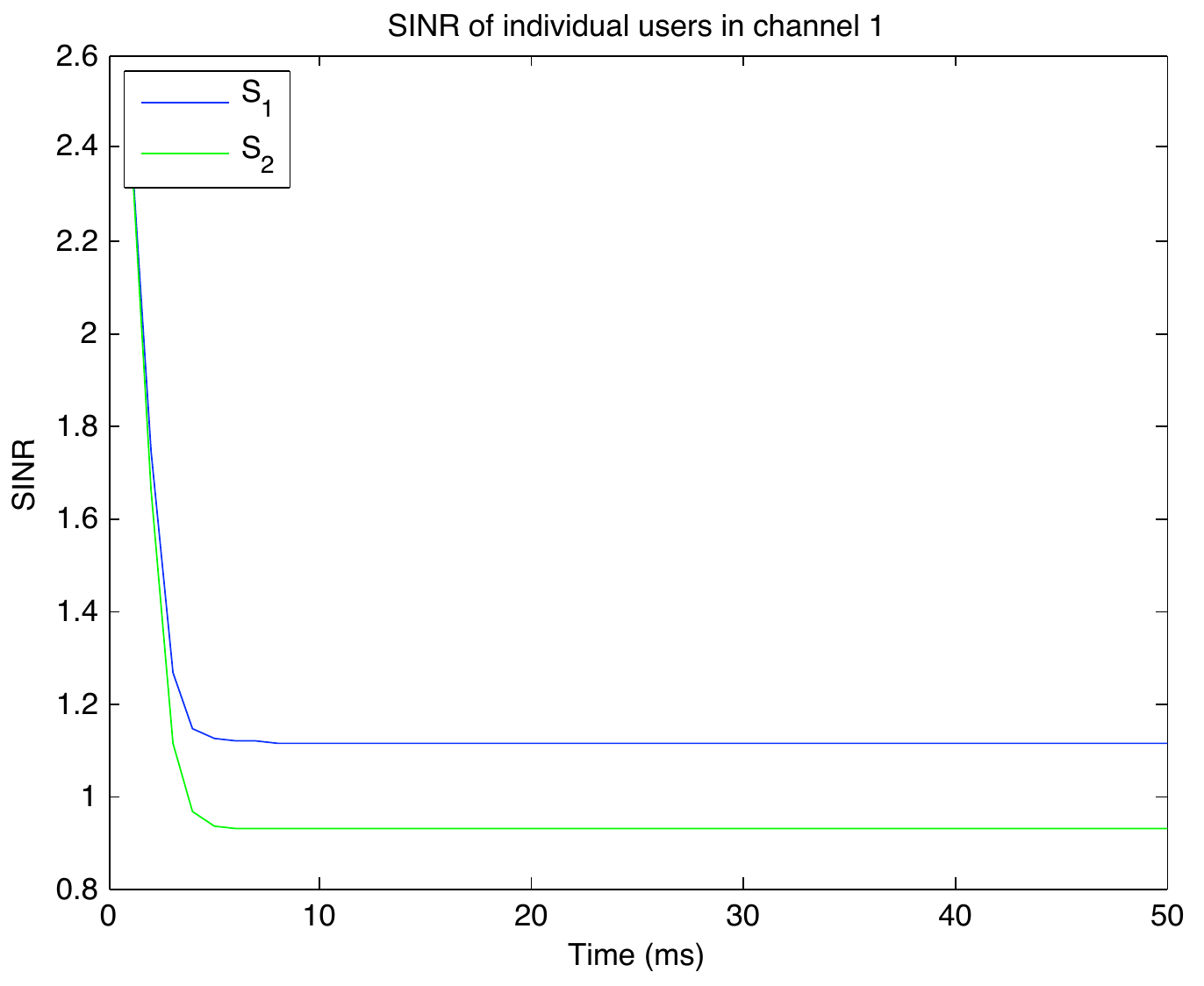}\hfill
    \includegraphics[width=3.2in]{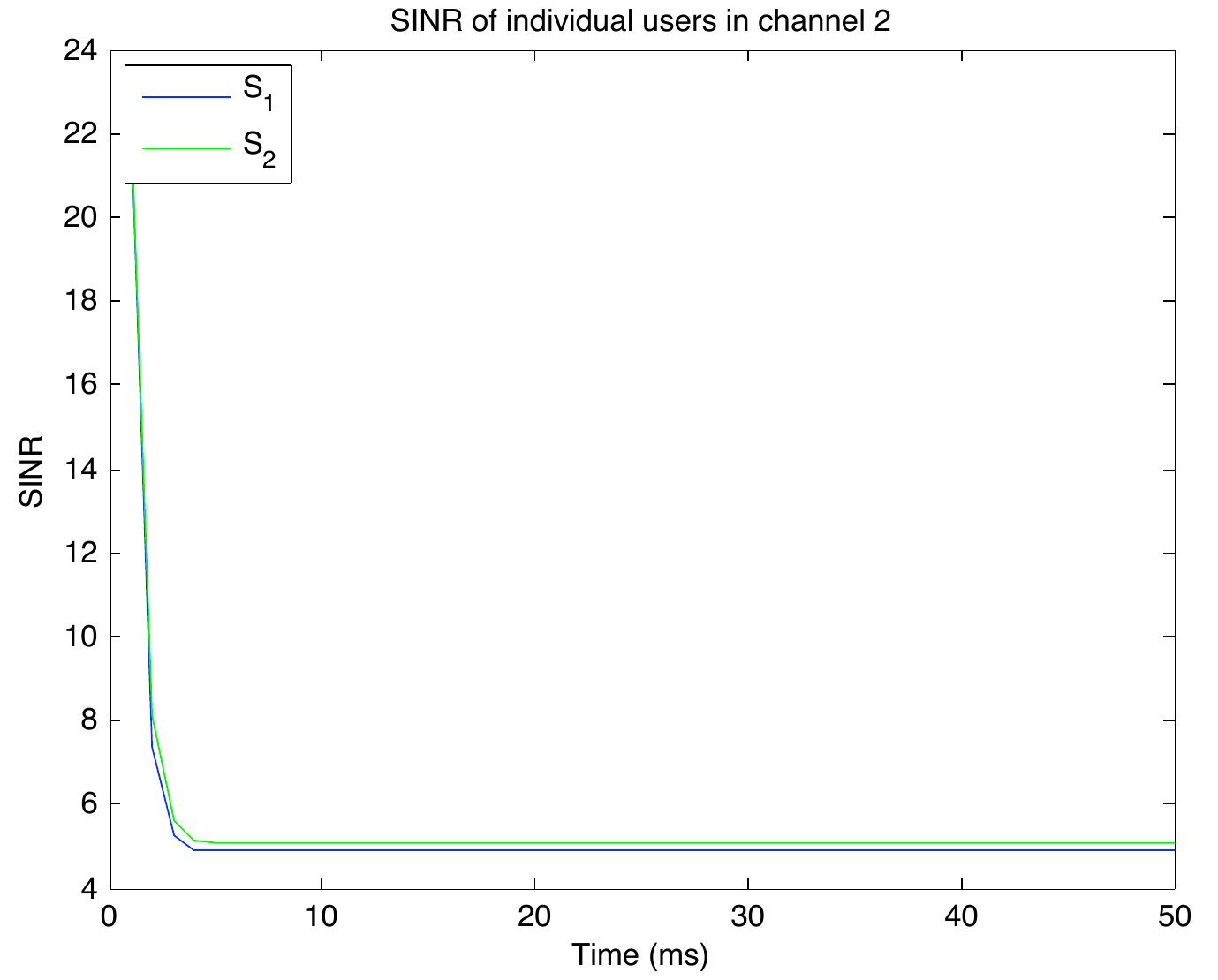}
    \caption{Power, allotted SINR target and SINR for each communication pair at each channel is shown. The allotted SINR converges fast to a feasible solution (not necessarily the optimal) and the system converges quite fast.}
    \label{chap5_ex0_all}
\end{figure*}


\subsection{Example 2: Infeasible network}

In this example, we consider the following network (Figure \ref{chap5_ex2}), which is the same as before, but the distance between the pairs is reduced now to $1.8$ meters. The interference between the two pairs is much bigger now and the SINR region for this network is non-convex (Figure \ref{chap5_ex2_sinrRegion}).

\begin{figure}[h]
    \centering
   \includegraphics[width=1.7in]{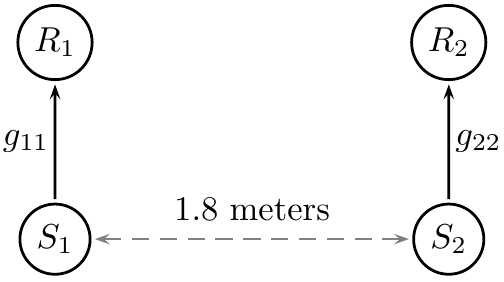}
    \caption{A wireless ad-hoc network of $n=4$ nodes, consisting of two communication pairs \{$S_{i}\rightarrow R_{i}$\}. The distance between the transmitters is 1.5 m.}
    \label{chap5_ex2}
\end{figure}

\begin{figure}[h]
    \centering
    \includegraphics[width=3.4in]{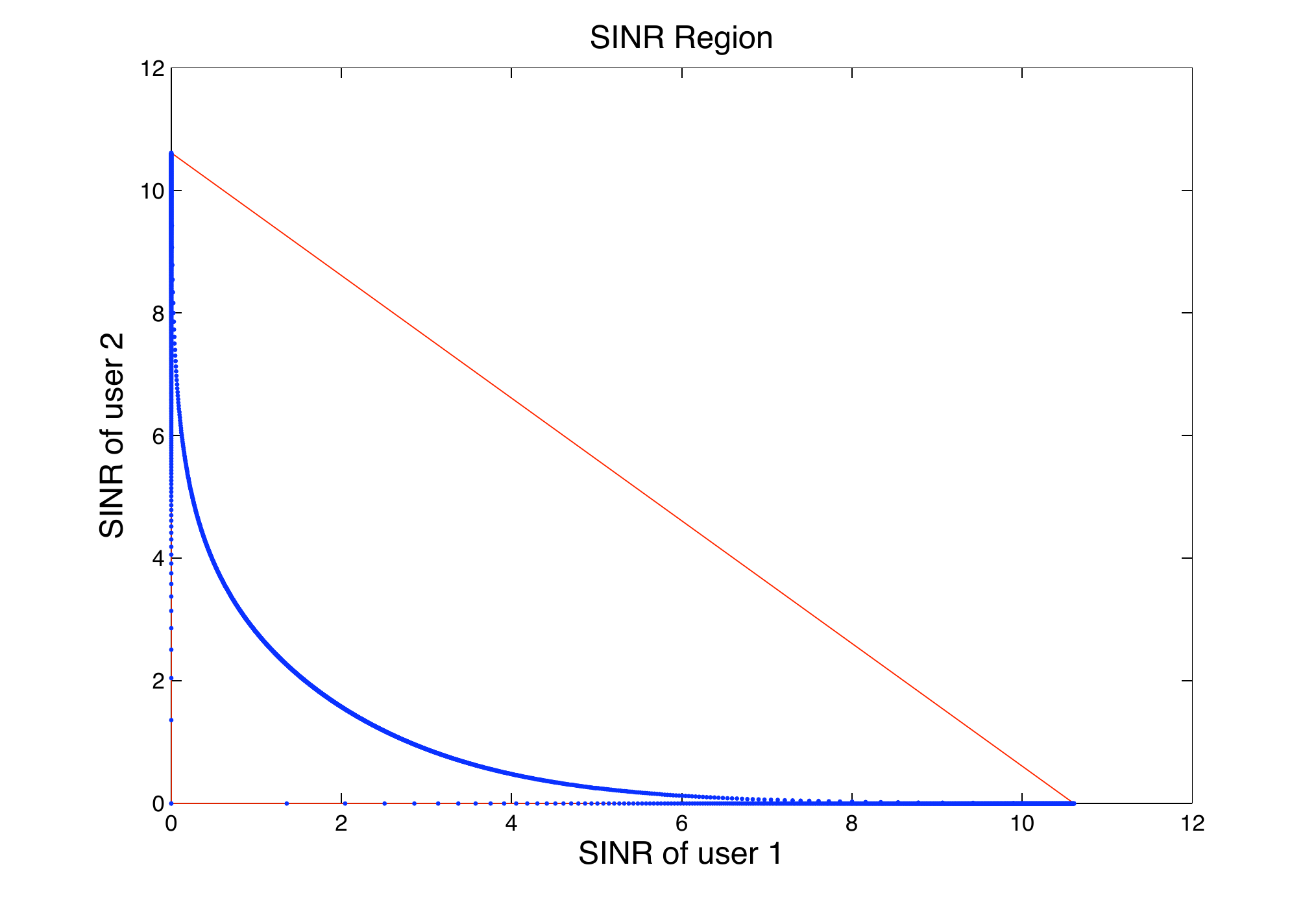}
    \caption{The SINR region of wireless ad-hoc network of Figure \ref{chap5_ex2}. It is non-convex and if the SINR QoS requirement is outside the SINR region, the users have to tend towards different channels.}
    \label{chap5_ex2_sinrRegion}
\end{figure}

As shown in Figure \ref{chap5_ex2_sinrRegion}, the SINR region is non-convex and hence if the SINR QoS requirement is outside the SINR region, then each user has to reduce the SINR requirement in the channel where the other increases it. In this way, they reduce the interference they cause to each other and they can eventually achieve the QoS requirements (Figure \ref{chap5_ex2_all}).

\begin{figure*}[h]
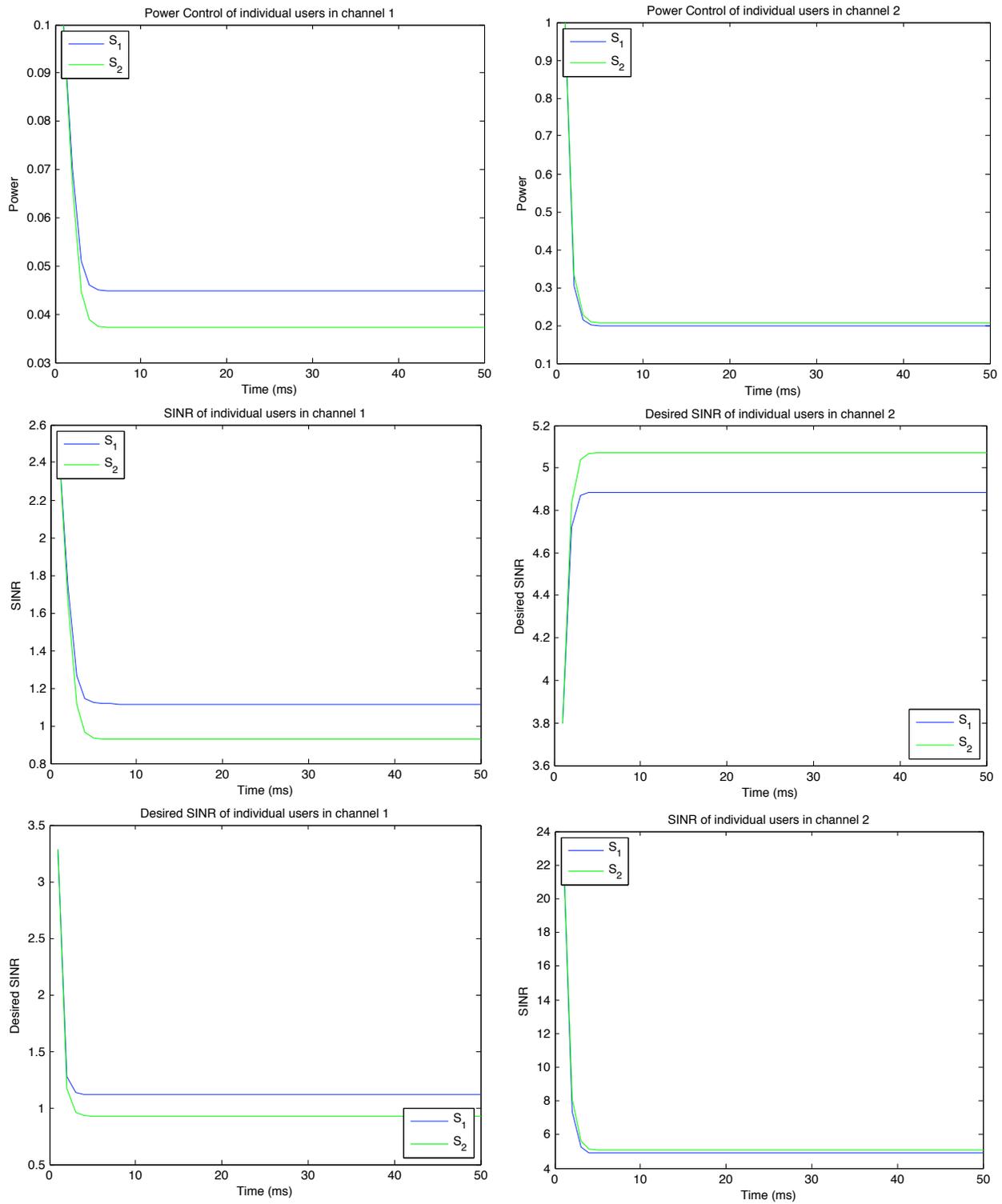

    \centering
    \includegraphics[width=3.2in]{POWERch1}\hfill
    \includegraphics[width=3.2in]{POWERch2}
    \includegraphics[width=3.2in]{SINRch1}\hfill
    \includegraphics[width=3.2in]{desiredSINRch2}
    \includegraphics[width=3.2in]{desiredSINRch1}\hfill
    \includegraphics[width=3.2in]{SINRch2}
    \caption{Power, allotted SINR target and SINR for each communication pair at each channel is shown. The allotted SINR converges slowly, since the QoS required cannot be achieved with simultaneous transmission.}
    \label{chap5_ex2_all}
\end{figure*}


%
%

\section{Conclusion}\label{conclusions}

While traditional approaches that tackle the problem of distributed power control consider a single channel only, here we investigated the more challenging, yet more promising power control problem over multiple channels. More specifically, we have developed a distributed algorithm consisting of two update formulae: one for the power and one for the allotted SINR target. In this work, we have considered multiple channels and aim to achieve the QoS requirement on average over all channels. Using standard Lyapunov stability theory we found the power control and allotted SINR target algorithms such that the whole network as a system converges to a feasible solution, if one exists. As aforementioned, the solution found is not necessarily the optimal. The results of this work are of paramount importance since wireless nodes can make use of multiple channels simultaneously and achieve a QoS on average, that otherwise would be impossible. Note that heterogeneous networks can be also considered in the sense that each node can be equipped with different number of antennae and hence have access to a different number of channels. Such algorithms are useful for elastic and/or opportunistic traffic, where time-varying rates are allowed and large delays are tolerated.

The convergence rate of this algorithm is an ongoing research. It is also important to specify the stability conditions of this algorithm in the presence of uncertainties (e.g. time-varying delays and changing environment).

\bibliographystyle{unsrt}
\bibliography{tau}

\end{document}